\documentclass[11pt,reqno]{amsart}
\usepackage{amssymb,epsfig,graphics}

\newtheorem{theorem}{Theorem}[section]
\newtheorem{proposition}[theorem]{Proposition}
\newtheorem*{proposition41}{Proposition 4.1}

\newtheorem{lemma}[theorem]{Lemma}
\newtheorem{sub-lemma}[theorem]{Sub-Lemma}
\newtheorem{remark}[theorem]{Remark}

\def\A{\mathcal{A}}
\def\L{\mathcal{L}}
\def\Z{\mathcal{Z}}

\def\V{\mathcal{V}}
\def\D{\Delta}
\def\NN{\mathbb{N}}
\def\PP{\mathbb{P}}
\def\RR{\mathbb{R}}

\def\ZZ{\mathbb{Z}}
\def\CC{\mathbb{C}}

\DeclareMathOperator{\supp}{supp}
\DeclareMathOperator{\diam}{diam}

\DeclareMathOperator{\length}{length}
\DeclareMathOperator{\cov}{Cov}

\let\eps=\varepsilon

\let\ds=\displaystyle
\def\glu{\!\!\!\!\!\!\!}
\def\gluu{\glu\glu\glu}
\def\gluuu{\gluu\gluu\gluu}
\def\glu{\!\!\!}

\def\D{\mathcal{D}}

\def\Z{\mathcal{Z}}
\def\RR{{\mathbb R}}
\def\1{{{\mathit 1} \!\!\>\!\! I} }
\renewcommand{\liminf}{\mathop{{\underline {\hbox{{\rm lim}}}}}}
\renewcommand{\limsup}{\mathop{{\overline {\hbox{{\rm lim}}}}}}
\DeclareMathOperator{\esssup}{esssup}

\begin{document}

\title{Back to balls in Billiards}
\author{Fran\c{c}oise P\`ene \and Beno\^\i t Saussol}
\address{1)Universit\'e Europ\'eenne de Bretagne, 
France\\
2)Universit\'e de Brest, laboratoire de
Math\'ematiques, CNRS UMR 6205, France\\
3)Fran\c{c}oise P\`ene is partially supported
by the ANR project TEMI (Th\'eorie Ergodique en
mesure infinie)}
\email{francoise.pene@univ-brest.fr}
\email{benoit.saussol@univ-brest.fr}
%\urladdr{}
\begin{abstract}
We consider a billiard in the plane with periodic configuration of convex scatterers.
This system is recurrent, in the sense that almost every orbit comes back arbitrarily close to the initial point. In this paper we study the time needed to get back in an $\eps$-ball about the initial point, in the phase space and also for the position, in the limit when $\eps\to0$.
We establish the existence of an almost sure convergence rate, and prove a convergence in distribution for the rescaled return times.
\end{abstract}
\date{December 19, 2008}
\maketitle
\bibliographystyle{plain}

\section{Introduction}

\subsection{Periodic Lorentz gas}
We consider a planar billiard with periodic 
configuration of scatterers. 
Such a model is also called a Lorentz process.
The motion of a free point particle 
bouncing on the scatterers according to 
Descartes' reflection law defines a flow.
The flow conserves the initial speed, so that 
without loss of generality we will assume that 
the particle moves with unit speed. 
This is a Hamiltonian flow which preserves a 
Liouville measure. Observe that the phase space is 
spatially extended and thus the measure is infinite.
We will suppose that the horizon is finite,
i.e. the time between two consecutive reflections
is uniformly bounded.

\begin{figure}
\begin{center}
  \includegraphics[scale=0.4]{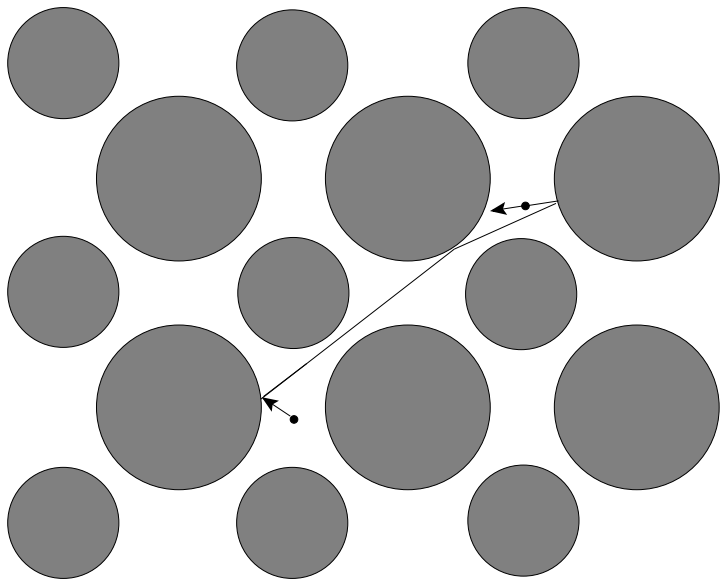}
\end{center}
\caption{Motion of a point particule in the Lorentz process}\label{fig:1}
\end{figure}

We are interested in the quantitative aspect of 
Poincar\'e's recurrence for the billiard flow.
It is known that this system is recurrent, in particular almost every orbit comes back arbitrarily close to the initial point . 
In this paper, our goal is to study the return time in balls, in the limit when the radius goes to zero.
Our main result is that 

(i) the time $Z_\eps$ to get back $\eps$-close to the initial point in the phase space is of order $\exp(\frac{1}{\eps^2})$ for Lebesgue almost all initial conditions

(i') the time ${\mathcal Z}_\eps$ to get back $\eps$-close to the initial position is of order $\exp(\frac{1}{\eps})$ for Lebesgue almost all initial conditions

(ii) we determine the fluctuations of 
$\eps^2\log Z_\eps$ and of $\eps\log {\mathcal Z}
_\eps$  by proving a convergence in distribution to 
a simple law.

This subject has been well studied recently in 
the setting of finite measure preserving 
transformations and typical behavior has been prove 
in a variety of chaotic systems: exponential statistics 
of return time, Poisson law, relation between recurrence rate and 
dimensions (see e.g.~\cite{abagal} for a state of the art in a probabilistic setting; also \cite{bs,colgalsch2}). 
The present work differs by two points from the existing literature.
First, the system in question has continuous time;
second, the main novelty is that its natural invariant measure is 
$\sigma$-finite. Very few works have appeared on 
the topic in this situation~\cite{bressaud-zweimueller,galkimpar,pene-saussol}.

A first reduction of the dynamics at the time of 
collisions with the scatterers (Poincar\'e section) 
and a second reduction by periodicity defines 
the praised billiard map. This map belongs to 
the class of hyperbolic systems with singularities. 
Since the work of Sina\"{i} \cite{Sin70} 
establishing the ergodicity of the billiard map,
it has been studied 
by many authors (let us mention~\cite{GO74}, \cite{BS80,BS81}, \cite{BCS90}
\cite{BCS91}) giving~: Bernoulli property,
central limit theorem. 
In the past ten years,
the new approach of L.-S. Young~\cite{young} 
has been exploited to get new significant 
results for the billiard map.
Among them, let us mention 
the exponential decay of 
correlations~\cite{young}, a new proof of the central 
limit theorem~\cite{young} and the local limit theorem
proved by Sz\'asz and Varj\'u~\cite{SV}.

Conze~\cite{Conze} and Schmidt~\cite{Schmidt}
proved that recurrence of the Lorentz process
follows from some central limit theorem
for the billiard map.
Sz\'asz and Varj\'u \cite{SV} used 
their local limit theorem
to give another proof of the recurrence. 
As proved by Sim\'anyi \cite{Simanyi} 
and the first named author \cite{FPCRAS}, 
once its recurrence proved,
it is not difficult to prove the
total ergodicity of the Lorentz process.
More recently, estimates on the first return time in the initial
cell have been established by Dolgopyat, Sz\'asz
and Varj\'u in \cite{DSV} and an analogous estimate
for the return time in the initial obstacle
follows from a paper of the first named author
\cite{FP08DCDS}.

\subsection{Precise description of the model and statement of the results}

We now precisely define the billiard flow $\Phi_t$.
Let $(O_i)_{i\in I}$ be a finite number of \textbf{open}, \textbf{convex} subsets
of $\RR^2$ with $C^3$ boundaries and \textbf{non-null
curvature}. We let $Q=\ds\RR^2\setminus \bigcup_{i\in I,\ell\in\ZZ^2}\ell+O_i$ be the billiard domain in the plane.
We suppose that the sets $\ell+O_i$ in this union have pairwise disjoint closure.
\begin{figure}
\begin{center}
  \includegraphics[scale=0.5]{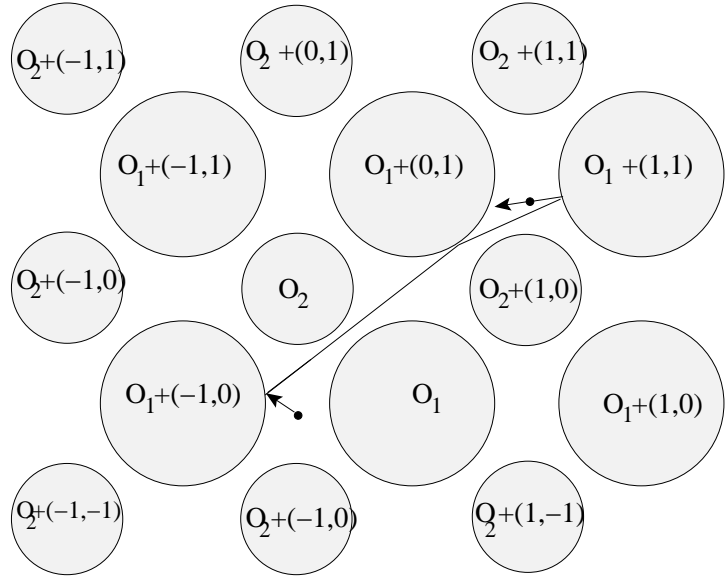}\end{center}
\caption{Labeling of the obstacles}\label{fig:label}
\end{figure}
The flow is given by the motion of a point particle with position $q\in Q$ and velocity $v\in S^1$. Namely, the motion is ballistic if there are no collisions with an obstacle in the time interval $[0,t]$: $\Phi_t(q,v)=(q+tv,v)$. At the time of a collision the velocity changes according to reflection law $v\mapsto v'$: If $n_q$ denotes the normal to the boundary of the obstacle at the point of collision $q\in\partial Q$, pointing inside the domain (i.e. outside the obstacle) then the angle $\angle(n_q,v')=\pi-\angle(n_q,v)$; see Figure~\ref{fig:billiard}.
\begin{figure}
\begin{center}
  \includegraphics[scale=0.5]{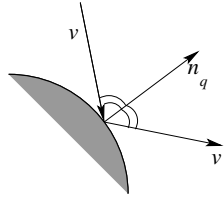}
\end{center}
\caption{Elastic reflection}\label{fig:billiard}
\end{figure}
We assume that the billiard has \textbf{finite horizon}, in the sense that the time between two consecutive collisions is uniformly bounded.

We endow the space $X=Q\times S^1$ with the product metric 
\[
d((q,v),(q',v'))=\max( d(q,q'),d(v,v')), 
\]
where for simplicity we denote all the distances by $d$. 
The flow preserves the Lebesgue measure on $Q\times S^1$; it is $\sigma$-finite but nevertheless the system is well known to be recurrent~\cite{Conze,Schmidt,SV}.

For $x\in X$ and $\eps>0$ we define the minimal time to get back $\eps$-close to the initial point by
\begin{equation}\label{eq:rec1}
Z_\eps(x):= \inf \left\{ t>\eps\colon d(\Phi_t(x),x)<\eps \right\}.
\end{equation}
The quantity $Z_\eps(\cdot)$ is well defined and finite for, at least, Lebesgue a.e. $x$.
We denote by $\Pi_Q:X=Q\times S^1\rightarrow Q$ the canonical projection.
We also define the minimal time to get back $\eps$-close to the initial position by
\begin{equation}\label{eq:rec2}
\Z_\eps(x):= \inf \left\{ t>\eps\colon d(\Pi_Q(\Phi_t(x)),\Pi_Q(x))<\eps \right\}.
\end{equation}
In the paper we give a precise asymptotic analysis of the return times $Z_\eps$ and $\Z_\eps$ expressed by our main theorem.
We say that a random variable $Y_\eps$ defined on $X$ converges in the strong distribution sense to a random variable $Y$ if for any probability $\PP \ll Leb$, 
$Y_\eps\to Y$ in distribution under $\PP$.

\begin{theorem}\label{thm:main}
The billiard flow satisfies

(i) for Lebesgue a.e. $x\in X$ we have $\ds\lim_{\eps\to0}\frac{\log\log Z_\eps(x)}{-\log\eps}=2$;

(ii) the random variable $\eps^2 \log Z_\eps$ converges as $\eps\to0$ in the strong distribution sense to a random variable $Y_0$ with distribution $P(Y_0>t)=\frac{1}{1+\beta_0 t}$ for some constant $\beta_0>0$;

(iii) for Lebesgue a.e. $x\in X$ we have 
$\ds\lim_{\eps\to0}\frac{\log\log {\mathcal Z}_\eps(x)}
{-\log\eps}=1$;

(iv) the random variable $\eps\log {\mathcal Z}_\eps$ converges as $\eps\to0$ in the strong distribution sense to a random variable $Y_1$ with distribution $P(Y_1>t)=\frac{1}{1+\beta_1 t}$ for some constant $\beta_1>0$.
\end{theorem}

\begin{remark}\label{rem:2}
The constant $\beta_0$ is equal to $\frac{2\beta}{\sum_{i\in I}\vert\partial O_i\vert}$, with $\beta=\frac{1}{2\pi\sqrt{\det\Sigma^2}}$ where $\Sigma^2$ is the asymptotic covariance matrix of the cell shift function $\kappa$ for the billiard map
$(\bar T,\bar\mu)$ defined by~\eqref{eq:cov}; See Section~\ref{sec:bilext} for precisions.
The constant $\beta_1$ is equal to $\frac{2\pi\beta}{\sum_{i\in I}|\partial O_i|}$.
\end{remark}

In Section~\ref{sec:maps} we define the billiard maps associated to our billiard flow.
In Section~\ref{sec:recbilmap} we investigate the behavior of return times for the billiard map. In Section~\ref{sec:bilext} we pursue this analysis for the extended billiard map, and building on the previous section we prove some preparatory results.
Section~\ref{sec:proofmain} is then devoted to the proof of the part of Theorem~\ref{thm:main} relative to returns in the phase space. Finally, in Section~\ref{sec:proofpos} we prove the part relative to returns for the position.

\section{Billiard maps}\label{sec:maps}

\subsection{Discrete time dynamics and new coordinates}

In order to study the statistical properties of the billiard flow, it is classical to make a \emph{Poincar\'e section} at collisions times, i.e. when $\Phi_t(q,v)\in\partial Q\times S^1$.
For definiteness, when $q\in\partial Q$ we choose the velocity $v$ pointing outside the obstacle, that is right after the collision. Denote for such a $q\in \partial Q$ and $v\in S^1$ by $\tau(q,v)$ the time before the next collision: $\tau(q,v)=\min\{t>0\colon \Phi_t(q,v)\in\partial Q\times S^1\}$.
Let $\phi$ be the Poincar\'e map: 
$\phi(q,v)=\Phi_{\tau(q,v)}(q,v)=(q',v')$ 
(see Figure~\ref{fig:bilmap}).
\begin{figure}
\begin{center}
  \includegraphics[scale=0.5]{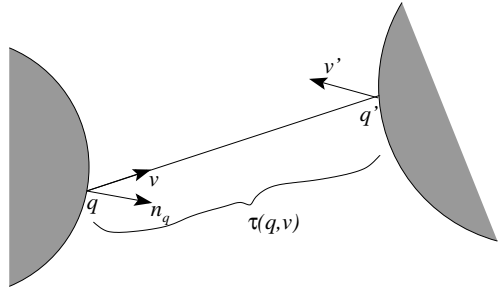}
\end{center}
\caption{The Poincar\'e section at collisions times}\label{fig:bilmap}
\end{figure}

Next, we make a \emph{change of coordinates} for the 
base map. For each obstacle $O_i$ we choose an arbitrary origin and parametrize its boundary $\partial O_i$ by 
counter-clockwise arc-length. 
The position $q\in\partial Q$ is represented by $(\ell,i,r)$ if $q\in\ell+\partial O_i$ and $r$ is the parametrization of the point $q$.
The normal of the boundary at each point $q$ is denoted by $n_q$ and the velocity $v$ is represented by its angle $\varphi\in[-\frac\pi2,\frac\pi2]$ with $n_q$.
Let 
\[
M=\bigcup_{\ell\in\ZZ^2}\bigcup_{i\in I} \left(\{(\ell,i)\}\times\RR/_{|\partial O_i|\ZZ}\times\left[-\frac\pi2,\frac\pi2\right]\right)
\]
endowed with the product metric. Denote by $\psi\colon M\to\partial Q\times S^1$ the change of coordinate, such that $\psi(\ell,i,r,\varphi)=(q,v)$.
The \emph{extended billiard map} $T\colon M\to M$ is the Poincar\'e map $\phi$ in these new coordinates: $T=\psi^{-1}\circ\phi\circ\psi$.
The flow $\Phi_t$ is conjugated to the special flow $\Psi_t$ defined over the map $T$ under the free flight function $\tau\circ\psi$. Let $M_\tau=\{(m,s)\in M\times\RR\colon 0\le s<\tau(\psi(m))\}$.
We denote by $\pi\colon M_\tau\to M$ the projection onto the base defined by $\pi(m,s)=m$ and extends the conjugation $\psi$ to $M_\tau$ by setting $\psi(m,s)=\Phi_s(\psi(m))$.

Let $\bar M$ be the subset of $M$ corresponding to the cell $\ell=0$.
We define  the \emph{billiard map} $\bar T\colon \bar M\to \bar M$ corresponding to the quotient map of $T$ by $\ZZ^2$ ; this is well defined by $\ZZ^2$-periodicity of the obstacles.
The \emph{cell shift} function $\kappa\colon M\to\ZZ^2$ is defined by 
$\kappa(\ell,i,r,\varphi)=\ell'-\ell$ if 
$T(\ell,i,r,\varphi)=(\ell',i',r',\varphi')$.

During the proof of our theorems on the billiard flow we will prove a version of the \emph{local limit theorem} for the billiard map suitable for our purpose, as well as a property of recurrence called \emph{exponential law for the return time} statistics. 

\subsection{Different quantities related to recurrence}

The notion of recurrence in these billiard maps gives rise to the definition of the following different quantities.
Let $m\in M$ and $\bar m\in\bar M$.

Let $W_A(m)$ be the first iterate $n\ge1$ such that $T^nm\in A$ for some subset $A\subset M$.

Let $\bar W_B(\bar m)$ be the first iterate $n\ge1$ such that $\bar T^n \bar m\in B$ for some subset $B\subset\bar M$.

Let $W_\eps(m)$ be the first iterate $n\ge1$ such that $d(T^nm,m)<\eps$ for some $\eps>0$.

Let $\bar W_\eps(\bar m)$ be the first iterate $n\ge1$ such that $d(\bar T^n\bar m,\bar m)<\eps$ for some $\eps>0$.

\section{Recurrence for the billiard map}\label{sec:recbilmap}

Recall that the billiard map $\bar T$ preserves a probability measure $\bar\mu$ equivalent to the Lebesgue measure on $\bar M$, whose density is given by
\[
\rho(\ell,i,r,\varphi)=\frac{1}{2\Gamma}\cos\varphi,
\quad\text{where}\quad
\Gamma:=\sum_{i\in I}|\partial O_i|.
\]
The billiard system $(\bar M,\bar T)$ is two dimensional with one negative and one positive Lyapunov exponent and the singularities are not too wild, therefore the result on recurrence rate \cite{saussol} applies.
\begin{theorem}[\cite{saussol}]\label{thm:rrbilliard}
The recurrence rate of the billiard map is equal to the dimension:
\[
\lim_{\eps\to0}\frac{\log \bar W_\eps}{-\log\eps}=2\quad\bar\mu\text{ a.e.}
\]
\end{theorem}

\begin{lemma} \label{lem:density}
For $\bar\mu$-almost every $m\in \bar M$,
for all $c_1>0$, $c_2>0$, $\alpha>0$
and for all family $(D_\varepsilon)_\varepsilon$ 
of sets containing $m$
such that $D_\varepsilon\subseteq B(m,c_2\varepsilon)$
and $\bar\mu(D_\varepsilon)\ge 
c_1(\diam(D_\varepsilon))^2$,
we have
\[
\bar\mu(\bar W_{B(m,\eps)}\le \eps^{-2+\alpha}|D_\eps)\to0.
\]
\end{lemma}
\begin{proof}
Let $\alpha>0$, $c_1>0$ and $c_2>0$.
Choose some $a\in(0,\alpha)$ and set for some $\eps_0>0$
\[
F_a=\{m\in\bar M \colon \forall\eps\le\eps_0,\frac{\log \bar W_\eps(m)}{-\log\eps}\ge2-a\}.
\]
By Theorem~\ref{thm:rrbilliard}
we have $\bar\mu(F_a)\to1$ as $\eps_0\to0$. 
There exists $\varepsilon_{1}>0$ such that,
for any $\eps<\eps_1$ we have the inclusions
\[
D_\eps\cap\{\bar W_{B(m,\eps)}\le 
\eps^{-2+\alpha}\}
\subset
D_\eps\cap\{\bar W_{(1+c_2)\eps}\le\eps^{-2+\alpha}\}
\subset 
D_\eps\cap F_a^c.
\]
Thus for any density point $m$ of the set $F_a$ relative 
to the Lebesgue basis given by $(B(\cdot,\eps))
_\eps$ we obtain
\[
\begin{split}
\bar\mu(\bar W_{B(m,\eps)}\le \eps^{-2+\alpha}|D_\eps)
&\le
\bar\mu(F_a^c|D_\eps)\\
&\le
\bar\mu(F_a^c | B(m,\diam D_\eps))\frac{\bar\mu(B(m,\diam D_\eps))}{\bar\mu(D_\eps)}
\to0
\end{split}
\]
as $\eps\to0$.
\end{proof}
We call \emph{non-sticky} a point $m$ satisfying the
conclusion of Lemma~\ref{lem:density}
and we denote by $\mathcal{NS}$ the set
of non-sticky points. We emphasize that $\bar\mu(\mathcal{NS})=1$

Next theorem says that the return times and entrance times in balls are exponentially distributed for the billiard map. 

\begin{theorem}\label{thm:billiardmap}
Let $m\in\mathcal{NS}$ be a non-sticky point. We have
\[
\begin{split}
\bar\mu(\bar\mu(B(m,\eps))\bar W_{B(m,\eps)}(\cdot)>t|B(m,\eps))&\to e^{-t}, \\
\bar\mu(\bar\mu(B(m,\eps))\bar W_{B(m,\eps)}(\cdot)>t)&\to e^{-t},
\end{split}
\]
uniformly in $t\ge0$, as $\eps\to0$.
\end{theorem}

We denote by $A^{[\eta]}$ the $\eta$-neighborhood of a set $A$.

\begin{proof}
We use an approximation by cylinders, the exponential mixing and the method developed in~\cite{hsv} for exponential return times and entrance times. 
We write $A=B(m,\eps)$ for convenience.
According to Theorem~2.1 in~\cite{hsv}, it suffices to show that 
\[
\sup_n \left|\bar\mu(\bar W_A>n|A)-\bar\mu(\bar W_A>n)\right| = o_\eps(1),
\]
since it will imply that the limiting distributions exist and are both exponential.

Let $c_3>0$ be such that $\mu(\partial A^{[\eta]})\le c_3\eta$ independently of $\eps$.
Let $k$ be an integer such that $\delta^k\approx\eps^3$. 
Let $g$ be an integer such that $\theta^{g-2k}\approx \eps^3$, where $\theta$ is
the constant appearing in Theorem~\ref{thm:young}.

If $m$ is a non-sticky point, observing that $g$ is logarithmic in $\eps$, we have for any integer $n$,
\[
\left|\bar\mu(\bar W_A>n|A)-\bar\mu(\bar W_A\circ \bar T^g>n-g|A)\right|\le \bar\mu(\bar W_A\le g|A)=o_\eps(1).
\]
Set $E=\{\bar W_A>n-g\}$.
We approach $A$ and $E$ by a union of cylinder sets:

Let $A'$ be the union of all the cylinders (see Appendix~\ref{sec:pf} for the precise definition) $Z\in\xi_{-k}^k$ such that $Z\subset A$.
We have $A'\subset A$ and $A\setminus A'\subset \partial A^{[c_0\delta^k]}$ by Lemma~\ref{lem:bord}. Thus we get $\bar\mu(A\setminus A')\le c_3 c_0\delta^k$.

Let
\[
E'=\bigcap_{j=1}^{n-g}\bar T^{-j}(\cup_{Z\in\xi_{-k-j}^{k+j},Z\cap A\neq\emptyset}Z)^c.
\]
We have $E'\subset E$ and by Lemma~\ref{lem:bord} again
\[
E\setminus E'\subset (\partial A)^{[c_0\delta^k]}\cup \bigcup_{j=1}^{n-g} \bar T^{-j}(\partial A)^{[c_0\delta^{k+j}]}.
\]
Thus by  the invariance of $\bar\mu$ we get
$\bar\mu(E\setminus E')\le c_3 c_0\frac{\delta^k}
{1-\delta}$.
Using the decay of correlations (for cylinders, see Theorem~\ref{thm:young} in Appendix~\ref{sec:pf}) we get that
\[
\left|\bar\mu(A'\cap \bar T^{-g}E')-\bar\mu(A')\bar\mu(E')\right|\le C\theta^{g-2k}=o(\bar\mu(A)).
\]
Furthermore, 
\[
\left|\bar\mu(\bar W_A>n)- \bar\mu(E)\right|\le \bar\mu(\bar W_A\le g) \le g\bar\mu(A) = o(1).
\]
Putting together all these estimates gives 
\[
\left|\bar\mu(\bar W_A>n|A)-\bar\mu(\bar W_A>n)\right| = o(1),
\]
uniformly in $n\in \NN$.
\end{proof}

Next, using the mixing property again we can condition on a smaller set and still get the same limiting law.
\begin{proposition}\label{pro:prop11bar}
For any $m\in\mathcal{NS}$ there exists a function $f_m$
such that $\lim_{\varepsilon\rightarrow 0}
f_m(\varepsilon)=0$ and such that
the following holds:

For any $\eps>0$ and any balls $D_\eps$, $A_\eps$ of $\bar M$ such that

(i) $m\in D_\eps\subset A_\eps=B(m,\eps)$,

(ii) $\bar\mu(D_\eps)\ge \eps^{2.25}$,

we have for any $n$
\[
\left|\bar\mu(\bar W_{A_\eps}(\cdot)>n|D_\eps) - e^{-n\bar\mu(A_\eps)}\right|\le f_m(\varepsilon).
\]
\end{proposition}

\begin{proof}
We approximate the sets $D$ and $E=\{\bar W_A>n\}$ from the inside by sets $D'$ and $E'$ as we approximated the sets $A$ and $E$ in the proof of Theorem~\ref{thm:billiardmap}. 
With the same $g$ we get 
\[
\left|\bar\mu(\bar W_A>n|D)-\bar\mu(\bar W_A\circ\bar T^g>n-g|D)\right|\le \bar\mu(\bar W_A\le g|D)=o(1).
\]
for non-sticky points.
Using the exponential decay of correlations for cylinders given by Theorem~\ref{thm:young} we get that 
\[
\begin{split}
\bar\mu(\bar W_A\circ\bar T^g>n-g|D)
&=\bar\mu(\bar W_A\circ\bar T^g>n-g)+o(1)\\
&=e^{-n\bar\mu(A)}+o(1)
\end{split}
\]
by Theorem~\ref{thm:billiardmap}.
\end{proof}

The following result of independent interest will not be used in the sequel an can be derived from Proposition~\ref{pro:prop11bar} as Proposition~\ref{pro:fluctext} would be derived from Proposition~\ref{pro:condflucext}. Therefore we omit its proof.

\begin{proposition} 
The random variable $4\eps^2\rho(\cdot)\bar W_\eps(\cdot)$ converges, in the strong distribution sense, to the exponential law with parameter one.

The random variable $\eps^2\bar W_\eps(\cdot)$ converges, under the law of $\bar\mu$, to a random variable $Y$ which is a continuous mixture of exponentials. More precisely $Y$ has distribution 
\[
\PP(Y>t)=\int_{\bar M}e^{-4t\rho}d\bar\mu.
\]
\end{proposition}
\section{Recurrence for the extended billiard map}\label{sec:bilext}

Recall that the extended billiard map $(M,T)$ preserves the $\sigma$-finite
measure $\mu$ equivalent to the Lebesgue measure on $M$, which is the image of the Lebesgue measure on $Q\times S^1$, whose density is equal to $\cos\varphi$. 
Note that
\begin{equation}\label{eq:const}
\mu|_{\bar M} = 2\Gamma \bar\mu.
\end{equation}

\subsection{Preliminary results on the extended billiard map}

We will use the following extension of Sz\'asz
and Varj\'u's local limit theorem \cite{SV}. For simplicity we use the notation 
$\bar\mu(A_1;\ldots;A_n)=\bar\mu(A_1\cap\cdots\cap A_n)$.
\begin{proposition}\label{pro:cullt}
Let $p>1$. There exists $c>0$ such that, 
for any $k\ge 1$, if $A\subset\bar M$ is a 
union of components of $\xi_{-k}^{k}$ and $B\subset \bar M$ is a union of $\xi_{-k}^\infty$ then for any $n > 2k$ and $\ell\in\ZZ^2$
\[
\left|
\bar\mu(A\cap \{S_n\kappa=\ell\}\cap \bar T^{-n}(B))
-\frac{\beta e^{-\frac1{2(n-2k)} (\Sigma^2)^{-1}\ell\cdot\ell}}{(n-2k)}\bar\mu(A)\bar\mu(B)\right|
\le \frac{ck\bar\mu(B)^{\frac1p}}{(n-2k)^\frac32}
\]
where $\beta=\frac{1}{2\pi\sqrt{\det \Sigma^2}}$.
\end{proposition}
The proof of Proposition~\ref{pro:cullt} is in Appendix~\ref{sec:cullt}.

\begin{proposition}\label{pro:upperbound}
Let $c_1,c_2,c_3$ and $c_4$ be some positive 
constants. For any $m\in\mathcal{NS}$ there exists a function $f_m$ such that $\lim_{\varepsilon\rightarrow 0}
f_m(\varepsilon)=0$ and such that the following holds:

For any $\eps>0$ and any subsets $D_\eps$, $A_\eps$ of $\bar M$ such that

(i) $m\in D_\eps\subset A_\eps$,

(ii) $c_1\eps^2\le\bar\mu(A_\eps)$ and $A_\eps\subset B(m,c_2\eps)$,

(iii) for any $\eta>0$, $\bar\mu( \partial A_\eps^{[\eta]} ) \le c_3\eta$, and also 
$\bar\mu( \partial D_\eps^{[\eta]} ) \le c_3\eta$,

(iv) $\bar\mu (D_\eps)\ge c_1(\diam(D_\varepsilon))^2$ 
and $\bar\mu(D_\eps)\ge c_4\eps^{2.25}$,

uniformly in $N\in(e^{\log^2\eps},e^{\frac{1}{\eps^{2.5}}})$ we have
\[
\bar\mu(W_{A_\eps}(\cdot)>N|A_\eps) = \frac{1+o_\eps(1)}{1+\log(N)\bar\mu(A_\eps)\beta}
\]
and
\[
\bar\mu(W_{A_\eps}(\cdot)>N|D_\eps) = \frac{1}{1+\log(N)\bar\mu(A_\eps)\beta}+o_\eps(1)
\]
where the error terms $o_\eps(1)$ is bounded by $f_m(\varepsilon)$.
\end{proposition}

\begin{lemma}\label{lem:upper}
Under the hypothesis of 
Proposition~\ref{pro:upperbound}, for all
$m\in\bar M$ (even those not belonging to 
${\mathcal NS}$), we have
\[
\bar\mu(W_A>N|D)+\beta\log(N)\bar\mu(A)\bar\mu(W_A>N|A)\le1+o_\eps(1),
\]
where the error term only depends on the positive 
constants $c_i$.
\end{lemma}

\begin{proof}
As used by Dvoretzky and Erd\"os in \cite{DE}, 
a partition of $D$ with respect to the last entrance time $q$ into the set $A$ in the time 
interval $[0,\ldots,N]$ gives
\begin{equation}\label{eq:de}
\begin{split}
\bar\mu(D)
&=\sum_{q=0}^N \bar\mu(D ; S_q\kappa=0 ; \bar T^{-q}(A\cap \{W_A>N-q\}))\\
&\ge\sum_{q=0}^N\bar\mu(D ; S_q\kappa=0 ; \bar T^{-q}(E))
\end{split}
\end{equation}
with $E=A\cap\{W_A>N\}$.

Let $k$ be such that $\delta^k\approx\eps^3$.
We approach $D$ and $E$ by cylindrical sets:

Let $D'$ be the union of cylinders $Z\in\xi_{-k}^k$ such that $Z\subset D$.
We have $D'\subset D$ and $D\setminus D'\subset \partial D^{[c_0\delta^k]}$ by Lemma~\ref{lem:bord}, thus by the hypothesis (iii) we get $\bar\mu(D\setminus D')\le c_3 c_0\delta^k$.

Let $A'$ be the corresponding cylindrical approximation for $A$ and set
\[
E'=A'\cap(\bigcap_{j=1}^N\left[\{S_j\kappa\neq0\}\cup \bar T^{-j}(\cup_{Z\in\xi_{-k-j}^{k+j},Z\cap A\neq\emptyset}Z)^c 
\right]).
\]
We have $E'\subset E$ and by Lemma~\ref{lem:bord}
\[
E\setminus E'\subset (\partial A)^{[c_0\delta^k]}\cup \bigcup_{j=1}^N \bar T^{-j}(\partial A)^{[c_0\delta^{k+j}]}.
\]
Thus by the hypothesis (iii) and the invariance of $\bar\mu$ we get
$\bar\mu(E\setminus E')\le c_3 c_0\frac{\delta^k}
{1-\delta}$.

Set $p_0\approx\eps^{-a}$ with $a=4.6>2\times2.25$.
By~\eqref{eq:de} and the inclusions we get
\[
\bar\mu(D)\ge \bar\mu(D\cap E)+\sum_{q=p_0}^N
\bar\mu(D';S_q\kappa=0;\bar T^{-q}E').
\]
It follows from Proposition~\ref{pro:cullt} that
\[
\bar\mu(D)\ge\bar\mu(D\cap E)+\sum_{q=p_0}^N\beta \frac{\bar\mu(D')\bar\mu(E')}{q-2k}-\sum_{q=p_0}^N\frac{ck}{(q-2k)^{\frac32}}.
\]
The error term is bounded by $\frac{ck}{\sqrt{p_0-2k}}=O(\log(\eps)\eps^{a/2})\ll c_4\eps^{2.25}\le\bar\mu(D)$. Thus, since $\log p_0=o(\log N)$,
\[
\bar\mu(D\cap E)+\beta\log(N)\bar\mu(D')\bar\mu(E')\le \bar\mu(D)(1+o(1)).
\]
Therefore, using $\bar\mu(D\setminus D')\le c_3c_0\delta^k=c_3c_0\eps^3\ll c_4\eps^{2.25}-c_3c_0\eps^3\le \bar\mu(D')$, we get
\[
\bar\mu(D\cap E)+\beta\log(N)\bar\mu(D)\bar\mu(E')\le \bar\mu(D)(1+o(1)).
\]
Notice that $\bar\mu(E\setminus E')\log N\le \frac{c_3c_0}{1-\delta}\delta^k\log N=o(1)$, from which it follows that
\[
\bar\mu(D\cap E)+\beta\log(N)\bar\mu(D)\bar\mu(E)\le \bar\mu(D)(1+o(1)).
\]
A division by $\bar\mu(D)$ yields, since $E=A\cap\{W_A>N\}$
and $D\subset A$,
\[
\bar\mu(W_A>N|D)+\beta\log(N)\bar\mu(A)\bar\mu(W_A>N|A)\le1+o(1)
\]
\end{proof}

\begin{lemma}\label{lem:lower}
Under the hypotheses of Proposition~\ref{pro:upperbound} we have 
\[
\bar\mu(W_A>N|D)+\beta\log(N)\bar\mu(A)\bar\mu(W_A>N|A) = 1+o_\eps(1).
\]
\end{lemma}

\begin{proof}
Let $\alpha\in(0,0.25)$ and set $M_\eps=\eps^{2(-1+\alpha)}$.
We use the same decomposition as in Equation~\eqref{eq:de} again, with $n_N=N\log(N)$ and $m_N=n_N-N$:
\[
\bar\mu(D)=\sum_{q=0}^{n_N}\bar\mu(D ; S_q\kappa=0 ; 
\bar T^{-q}(A\cap\{W_A>n_N-q\})).
\]
We divide this sum into four blocks: $S_0$ is the term for $q=0$, $S_1$ is the sum for $q$ in the range $1,\ldots,M_\eps$,
$S_2$ in the range $M_\eps+1,\ldots,m_N$ and
$S_3$ in the range $m_N+1,\ldots,n_N$.

The value of $S_0$ is simply
\[
S_0=\bar\mu(D ; W_A>n_N) \le \bar\mu(D ; W_A>N).
\]
By assumption (conclusion of Lemma~\ref{lem:density}), we have
\[
S_1=\bar\mu(D ; W_A\le M_\eps)
   \le\bar\mu(D;\bar W_{B(m,c_2\eps)}\le M_\eps) 
  = o(\bar\mu(D)).
\]
When $q\le m_N$ we have $n_N-q\ge N$, therefore we have
\[
S_2\le\sum_{q=M_\eps+1}^{m_N}\bar\mu(D ; S_q\kappa=0 ; \bar T^{-q}(E))
\]
with $E=A\cap\{W_A>N\}$. Let $k$ be such that $\delta^k\approx\eps^3$.
We approximate the sets $D$ and $E$ by cylinders: let $D''$ be the union of cylinders $Z\in\xi_{-k}^k$ such that $Z\cap D\neq\emptyset$.
Let $A''$ be the corresponding enlargement for $A$ and let
\[
E''=A''\cap\bigcap_{j=1}^N\left[\{S_j\kappa\neq0\}
\cup \bar T^{-j}(\cup_{Z\in\xi_{-k-j}^{k+j},
Z\subseteq A}Z)^c\right].
\]
We have $D\subset D''$ and by Lemma~\ref{lem:bord}, $D''\setminus D\subset (\partial D)^{[c_0\delta^k]}$.
Thus by Hypothesis~(iii) we get that $\bar\mu(D''\setminus D)\le c_3c_0\delta^k$.
Similarly, $E\subset E''$ and $E''\setminus E
\subset (\partial A)^{[c_0\delta^k]}\cup
\bigcup_{j=1}^N\bar T^{-j}(\partial A)
^{[c_0\delta^{k+j}]}$.
Thus by hypothesis~(iii) we get that 
$\bar\mu(E''\setminus E)\le c_3c_0\frac{\delta^k}
{1-\delta}$ and so $\log(m_N)\bar\mu(E''\setminus E)
=o(1)$.
By Proposition~\ref{pro:cullt} with $p$ such that
$1+\frac{2}{p}>2.5$, we get
\[
\begin{split}
S_2
&\le\sum_{q=M_\eps+1}^{m_N}\bar\mu(D'' ; S_q\kappa=0 ; \bar T^{-q}(E''))\\
&\le 
\sum_{q=M_\eps+1}^{m_N} \left[\beta\frac{\bar\mu(D'')
\bar\mu(E'')}{q-2k}+\frac{ck\bar\mu(E'')^{\frac1p}}
{(q-2k)^{\frac32}}\right]\\
&\le \log(m_N)\beta\bar\mu(D'')\bar\mu(E'')+\frac{ck\bar\mu(A'')^{\frac{1}{p}}}{\sqrt{M_\eps-2k}}\\
&\le
\log(m_N)\beta\bar\mu(D)\bar\mu(E)(1+o(1))
+o(\bar\mu(D))+ O(\log(\eps)\eps^{1-\alpha}\eps^{2/p})
\end{split}
\]
The last error term is $o(\bar\mu(D))$ provided $1-\alpha+2/p>2.25$.
In addition $\log m_N\sim\log N$, hence
\[
S_2\le\beta\log(N)\bar\mu(D)\bar\mu(E)(1+o(1))+o(\bar\mu(D)).
\]
Finally, by Proposition~\ref{pro:cullt} we get
\[
\begin{split}
S_3
&\le\sum_{q=m_N}^{n_N}\bar\mu(D'' ; S_q\kappa=0 ; \bar T^{-q}A'')\\
&\le \sum_{q=m_N}^{n_N} \left[\beta\frac{\bar\mu(D'')\bar\mu(A'')}{q-2k}+\frac{ck}{(q-2k)^{\frac32}}\right]\\
&\le
\beta\log\left(\frac{n_N}{m_N}\right)\bar\mu(D)\bar\mu(A)(1+o(1))+\frac{k}{\sqrt{m_N-2k}}
\end{split}
\]
Moreover $\log(\frac{n_N}{m_N})=o(1)$,  
and the last error term is again $o(\bar\mu(D))$ since $m_N\ge N$.

We conclude that 
\[
(1+o(1))\bar\mu(D)\le \bar\mu(D\cap E)+\beta\log(N)\bar\mu(D)\bar\mu(E).
\]
A division by $\bar\mu(D)$ yields, since $E=A\cap\{W_A>N\}$ and $D\subset A$,
\[
\bar\mu(W_A>N|D)+\beta\log(N)\bar\mu(A)\bar\mu(W_A>N|A)\ge 1-o(1).
\]
The reverse inequality also holds by Lemma~\ref{lem:upper}, finishing the proof.
\end{proof}

\begin{proof}[Proof of Proposition~\ref{pro:upperbound}]
Lemma~\ref{lem:lower} with $D=A$ gives us
\begin{equation}\label{eq:D=A}
\bar\mu(W_A>N|A)=\frac{1+o(1)}{1+\beta\log(N)\bar\mu(A)}.
\end{equation}
This proves the proposition in the special case $D=A$.
We turn now to the general case.
Applying Lemma~\ref{lem:lower} again, together with~\eqref{eq:D=A} we get, 
\[
\bar\mu(W_A>N|D)+\beta\log(N)\bar\mu(A)\frac{1+o(1)}{1+\beta\log(N)\bar\mu(A)} = 1+o(1)
\] 
which proves the proposition.
\end{proof}

\subsection{Recurrence results for the extended billiard map}\label{sec:billiardmap}
\begin{proposition}\label{pro:rrextended}
The recurrence rate for the extended billiard map is given by
\[
\lim_{\eps\to0}\frac{\log\log W_\eps}{-\log\eps}=2\quad\mu\text{-a.e.}
\]
\end{proposition}

\begin{proof}
Note that by $\ZZ^2$-periodicity it suffices to prove the statement $\bar\mu$ a.e. in $\bar M$.

\underline{Upper bound~:}
Let $\delta>0$ and set 
\[
\bar M_\delta=\{m\in\mathcal{NS}\colon \rho(m)>\delta\text{ and } \sup_{\eps\le\delta}f_m(\eps)\le 1\},
\]
where the function $f_m(\eps)$ appears in Proposition~\ref{pro:upperbound}.
Let us notice that there exist constants $c_i$
for which the hypotheses of Proposition~\ref{pro:upperbound} 
are satisfied for any $D_\eps=A_\eps=B(m,\eps/2)$, with $m\in \bar M_\delta$.
Let $\alpha\in(0,\frac{1}{2})$, $n\ge 1$ and $\varepsilon_n=\log^{-\alpha}n$. 
Take a cover of $\bar M_\delta$ by some sets
 $B(m,\varepsilon_n/2)$, $m\in{\mathcal P}_n\subset\bar M_\delta$ such that
$\#{\mathcal P}_n=O((\eps_n)^{-2})$.
According to Proposition~\ref{pro:upperbound}, we have

\begin{eqnarray*}
\bar\mu(\{W_{\varepsilon_n}\ge n\}\cap \bar M_\delta)&\le&\sum_{m\in\mathcal{P}_n}\bar\mu\left(
W_{B(m,\frac{\varepsilon_n}{2})}\ge n\left\vert
    B(m,\frac{\varepsilon_n}{2})\right.\right)\bar\mu\left(B(m,\frac{\varepsilon_n}{2})\right)\\
&\le&O((1+\beta c_1\log^{1-2\alpha}n)^{-1}).
\end{eqnarray*}
Now, by taking $n_k=\exp(k^{2/(1-\alpha)})$ and according
to the Borel-Cantelli lemma,
we get that, for almost all $m$ in $\bar M_\delta$, there exists $N_m$ such that, for any
$k\ge N_m$, $W_{\varepsilon_{n_k}}(m)< {n_k}$ and hence
$$\limsup_{k\rightarrow +\infty} \frac{\log\log W_{\varepsilon_{n_k}}(m)}
{-\log\varepsilon_{n_k}}\le\frac{1}{\alpha}.$$
Since $\log\varepsilon_{n_k}
\sim{\log\varepsilon_{n_{k+1}}}$, we get that $\bar\mu$-a.e. on $\bar M_\delta$
$$\limsup_{\eps\to0} \frac{\log\log W_{\varepsilon}}{-\log\varepsilon}\le \frac1\alpha.$$
We conclude that almost everywhere in $\bar M$, we have
$$\limsup_{\eps\rightarrow 0} 
\frac{\log\log W_{\eps}}{-\log\eps}\le 2.$$

\underline{Lower bound~:}
Let $\alpha>1/2$. Let $n\ge 1$ and
$\eps_n=\log^{-\alpha}n$.
We consider a cover of $\bar M$ by 
balls $B(m,\varepsilon_n)$ for $m\in{\mathcal P}'_n$ 
such that $\#{\mathcal P}'_n=O( \eps_n^{-2})$.
Let $k$ be such that $\delta^k\approx\eps_n^5$.
For each $m\in \mathcal{P}'_n$ we consider the sets $B''_{m}$ and $C''_{m}$ constructed from $B(m,\varepsilon_n)$
and $B(m,2\varepsilon_n)$ (respectively) like $A''$ was constructed from $A$
in the proof of Lemma~\ref{lem:lower}.

Applying Proposition~\ref{pro:cullt} we get
\begin{eqnarray*}
\bar\mu(d(\cdot,T^n(\cdot))<\eps_n)
&\le&\sum_{m\mathcal{P}'_n}
\bar\mu( B(m,\varepsilon_n)\cap\{S_n=0\}\cap\bar T^{-n}(B(m,2\varepsilon_n)))\\
&\le&\sum_{m}
\bar\mu( B''_{m}\cap\{S_n=0\}\cap\bar T^{-n}(C''_{m}))+O(\delta^k)\\
&\le& \sum_{m}\frac
{\bar\mu( B''_{m})\bar\mu(C''_{m})}
{n-2k}+\frac{ck}{(n-2k)^{3/2}}+O(\delta^k)\\
&\le& O\left(n^{-1}\log^{-2\alpha}n\right).
\end{eqnarray*}
Hence, according to the first Borel Cantelli lemma,
for almost every $m\in \bar M$,
there exists $N_m$ such that, for all $n\ge N_m$,
we have 
\[
d(m,T^n(m))\ge \eps_n.
\]
Let $u=\min(d(m,T^n(m)),\
   n=1,...,N_m)$. Note that $u>0$, otherwise we would have $m=T^p(m)$ for some $p$ and hence $m=T^n(m)$ infinitely often, which would contradict $d(m,T^n(m))\ge \eps_n$.
For all $n\ge N_m$ such that $\eps_n<u$ we have $W_{\eps_n}(m)\ge n$.
Hence
$$\liminf_{n\rightarrow +\infty} \frac{\log\log
    W_{\eps_n}(m)}{ -\log\eps_n}\ge \frac1\alpha.$$
Since $\log\eps_n\sim\log\eps_{n+1}$ we get $\displaystyle\liminf_{\eps\rightarrow 0}
\frac{\log\log{W}_{\eps}}
{-\log\eps}\ge \frac1\alpha$ almost everywhere on $\bar M$.
Therefore
$\displaystyle\liminf_{\eps\rightarrow 0}
\frac{\log\log{W}_{\eps}}
{-\log\eps}\ge 2$ $\bar\mu$-a.e.
\end{proof}

\begin{proposition}\label{pro:condflucext}
For a.e. $m\in \bar M$, and sequences of sets $(A_\eps)$ and $(D_\eps)$ such that the hypotheses (i)--(iv) of Proposition~\ref{pro:upperbound} are satisfied we have
\[
\mu(W_{A_\eps}>\exp(\frac{t}{\bar\mu(A_\eps)})|D_\eps)\to \frac{1}{1+t\beta}
\quad \text{as $\eps\to0$}.
\]
\end{proposition}
\begin{proof}
Proposition~\ref{pro:upperbound} with $N=\exp(\frac{t}{\bar\mu(A_\eps)})$ immediately gives the result.
\end{proof}
Note that in particular the proposition applies to the sequence of balls $A_\eps=D_\eps=B(m,\eps)$. This is the corresponding result to that of Theorem~\ref{thm:billiardmap} in the case of the extended billiard map.

\begin{proposition}\label{pro:fluctext} 
The random variable $4\eps^2\rho(\cdot)\log W_\eps(\cdot)$ converges in the strong distribution sense, to a random variable with law $P(Y>t)=\frac{1}{1+\beta t}$.
\end{proposition}
\begin{proof}
The proof is similar to that of Theorem~\ref{thm:main}-(ii), without the flow direction; See Section~\ref{sec:proofmain} for details.
Since it is an obvious modification of it and since this result will not be used in the sequel, 
we omit its proof.
\end{proof}

\section{Proof of the main theorem: recurrence in the phase space}\label{sec:proofmain}

We prove in this section Theorem~\ref{thm:main}-(i)and (ii) about the return times in the phase space $Z_\eps$ defined by~\eqref{eq:rec1}.

\subsection{Almost sure convergence: the first statement}
By $\ZZ^2$-periodicity it is sufficient to prove the result on $\bar M$.
Let $m\in \bar M$ be a point which is not on a singular orbit of $T$ and such that $W_\eps(m)$ follows the limit given by Proposition~\ref{pro:rrextended}.
By regularity of the change of variable $\psi$ (away from the singular set) there
exist two constants $0<a<b$ such that, for any $0\le s\le \tau(m)$, we have
\begin{equation}\label{eq:mapflow}
(\min\tau) (W_{b\eps}(m)-1) \le Z_\eps(\Phi_s\psi(m))\le (\max\tau) W_{a\eps}(m)
\end{equation}
since the free flight function $\tau$ is bounded from above and from below.
This implies the result for all the points $\Phi_s\psi(m)$.
By Fubini's theorem this concerns a.e. points in $Q\times S^1$, which proves the first statement.

\subsection{Convergence in distribution: the second statement} 
Unfortunately we cannot exploit the relation~(\ref{eq:mapflow}) above anymore. The problem is not with the multiplicative factor coming from $\tau$, but the fluctuations are sensible to the constants $a$ and $b$ and a direct method could only lead to rough bounds in terms of these constants.

The following lemma gives the measure of the projection of a ball $B(x,\eps)$ onto $M$.
\begin{lemma}\label{lem:mesureboule} For any $x\in X$ and $\eps>0$ such that the ball $B(x,\eps)$ does not intersect the boundary $\partial Q\times S^1$, we have
\[
\mu(\pi \psi^{-1} B(x,\eps)) = 4\eps^2.
\] 
\end{lemma}
\begin{proof}
Let $x=(q_0,v_0)\in X$. 
We consider the ball $B(q_0,\eps)$ as a new obstacle added in our billiard domain.
Let 
\[
\Delta_\eps:=\left\{ (q,v)\in Q\times S^1\colon q\in \partial B(q_0,\eps),\, |\angle(v_0,v)|<\eps,\, 
\langle n_q,v\rangle>0\right\}.
\]
Since the billiard map preserves the measure $\cos\varphi dr d\varphi$, we have 
\[
\mu(\pi \psi^{-1} B(x,\eps)) = 
\int_{\Delta_\eps} \cos\angle (n_q,v)\, dq dv.
\]
For any $v$ such that $|\angle(v_0,v)|<\eps$ a classical computation gives
\[
\int_{\{q:(q,v)\in\Delta_\eps\}}\cos \angle(n_q,v) \, dq= 2\eps,
\]
whence the result.
\end{proof}

Let $\PP=hd\L$ be the probability measure on $X$ under which we will compute the law of $Z_\eps$. Let $\bar X=\psi(\pi^{-1}\bar M)$. 
By $\ZZ^2$-periodicity, $Z_\eps$ has the same distribution under $\PP$ as under $\bar\PP=\bar hd\L$ where $\bar h(\cdot)=\sum_{\ell\in\ZZ^2} h(\cdot+\ell) 1_{\bar X}$.
Therefore we suppose that $\supp h\subset \bar X$.

Assume for the moment that the density $h$ is continuous and compactly supported in the set $\bar X'=\bar X\setminus (\psi(\bar M\times\{0\}\cup\pi^{-1}R_0))$, where $R_0=\{\varphi=\pm\frac\pi2\}$.
Then for any $r>0$ sufficiently small we have
\[
\supp h\subset \bar X_r:=\{\Phi_s(\psi(m))\colon m\in \bar M, r<d(m,R_0),r\le s\le \tau(m)-r\}.
\]
Let $K\subset \mathcal{NS}$ be a set of points where the convergence in Proposition~\ref{pro:condflucext} is uniform and such that 
\[
\PP(\{\Phi_s(\psi(m))\colon m\in K, 0\le s<\tau(m)\})>1-r.
\]
For any $\eps\in(0,r)$ sufficiently small, the $\eps$-neighborhood of $\bar X_r$ is contained in $\bar X$.

Let $\nu_\eps=\eps^{5/4}$.
Choose a family of pairwise disjoint open balls of radius $\nu_\eps$ in $\bar M$ such that their union has $\bar\mu$-measure larger than $1-4\nu_\eps$. We drop all the balls not intersecting $K$ and call $\{D_i\}$ the remaining family. For each $i$ we choose a point $m_i\in D_i\cap K$.
For each $i$, we take the family of times $s_{ij}=j\nu_\eps\in(0,\min_{D_i}\tau)$. Let
\[
P_{ij} =\{\Phi_s\psi(D_i)\colon s_{ij}\le s\le s_{ij}+\nu_\eps\}.
\]
We finally drop the $P_{ij}$'s not intersecting $\bar X'\cap \psi\pi^{-1}K $. 
Set $y_{ij}=\Phi_{s_{ij}}(\psi(m_i))$.
We have 
\begin{equation}\label{eq:sum}
\begin{split}
\PP\left(Z_\eps>\exp\left(\frac{t}{4\eps^2}\right)\right)
&\approx \pm r + \sum_{i,j}\PP\left( Z_\eps>\exp\left(\frac{t}{4\eps^2}\right) ; P_{ij}\right)\\
&\approx \pm r +\sum_{i,j}h(y_{ij})\L\left(Z_\eps>\exp\left(\frac{t}{4\eps^2}\right) ; P_{ij}\right)
\end{split}
\end{equation}
by uniform continuity of $h$.
Let 
\[
\begin{split}
A_{ij}^\pm&=\left\{m\in \bar M\colon\exists 0\le s<\tau(m) \text{ s.t. } \Phi_s(\psi(m))\in B(y_{ij},\eps\pm\nu_\eps) \right\}\\
&=\pi\psi^{-1}B(y_{ij},\eps\pm\nu_\eps)
\end{split}
\] 
denotes the projection onto the base of the balls. Let $\tau_-=\min\tau$ and $\tau_+=\max\tau$.
For any $x\in P_{ij}$, setting $m=\pi\psi^{-1}x\in \bar M$ its projection, we have
\begin{equation}\label{eq:zeps}
(\tau_-) (W_{A_{ij}^+}(m)-1)\le Z_\eps(x)\le (\tau_+) W_{A_{ij}^-}(m).
\end{equation}
Hence we have for any real $t>0$
\begin{equation}\label{eq:condi}
\nu_\eps\mu((\tau_-) (W_{A_{ij}^+}-1)>t ; D_i)\le
\L(Z_\eps>t ; P_{ij})
\le \nu_\eps\mu((\tau_+)W_{A_{ij}^-}>t ; D_i)
\end{equation}
Using the regularity of the projection $\pi$ on $X_r$, we see that the sets $A_{ij}^\pm$ fulfill the hypotheses of Proposition~\ref{pro:condflucext} with uniform constants.
Moreover, by Lemma~\ref{lem:mesureboule} and the relation~\eqref{eq:const}, we have 
\[
\bar\mu(A_{ij}^\pm)=\frac{4(\eps\pm\nu_\eps)^2}
{2\Gamma}.
\]
Therefore by our choice of the $m_i$'s, the difference
\[
\left|\mu\left((\tau_\mp)W_{A_{ij}^\pm}>\exp\left(\frac{t \Gamma}{2\eps^2}\right)\Big|D_i\right)
      -\frac{1}{1+\beta t}\right|
\]
tends to zero uniformly as $\eps\to0$.
Putting it together with~\eqref{eq:condi} in the computation~\eqref{eq:sum} yields to
\[
\limsup_{\eps\to0}\left|\PP\left(Z_\eps>\exp\left(\frac{t\Gamma}{2\eps^2}\right)\right)-\frac{1}{1+\beta t}\right|
\le r.
\]
Letting $r\to0$ gives the conclusion for a continuous density compactly supported on $\bar X'$.
The conclusion follows by an approximation argument, since any density $h\in L^1(\bar X,\L)$ may be approximated by a sequence $h_n$ of such densities.
\section{Proof of the main theorem: recurrence for the position}\label{sec:proofpos}

In this section we prove Theorem~\ref{thm:main}-(iii) and (iv) about the return times $\Z_\eps$ defined by~\eqref{eq:rec2}. The proof follows the scheme of the previous section but has additional arguments. We will detail the differences and indicate the common points.

We recall that $\Pi_Q$ is the canonical projection
from $X=Q\times S^1$ onto $Q$.
We will use the first return time 
$\overline{\mathcal Z}_\eps$
in the $\eps$-neighborhood of the initial
position modulo ${\mathbb Z}^2$ defined by
\[\overline{\mathcal Z}_\eps(x)=\min\left\{t>\eps\colon \Phi_t(x)
    \in \bigcup_{\ell\in\ZZ^2}
    B(\Pi_Q(x)+\ell,\eps)\times S^1\right\}. 
\]
For any $q$ in $Q$ and any $\eps>0$, we define
the backward projection of $B_\eps(q)\times
S^1$ on $M$ and on $\bar M$ by
\[
\begin{split}
\A_\eps(q)&=\left\{m\in M\colon \exists s\in
[0,\tau(\psi(m))),\ \Phi_s(\psi(m))\in B(q,\eps)
\times S^1\right\}, \\
\bar\A_\eps(q)=&\left\{m\in \bar M\colon \exists s\in
[0,\tau(\psi(m))),\ \Phi_s(\psi(m))\in 
  \bigcup_{\ell\in{\mathbb Z}^2}B(q+\ell,\eps)
\times S^1\right\}.
\end{split}
\]
\begin{lemma}\label{lem:mesAeps}
For any $q\in Q$ and any $\eps\in(0,d(q,\partial Q))$,
we have $\mu(\A_\eps(q))=4\pi\eps$ and so
$\bar\mu(\bar\A_\eps(q))=\frac{2\pi\eps}{\Gamma}$.
\end{lemma}
\begin{proof}
Indeed, since the measure $\cos(\varphi)drd\varphi$
is preserved by billiard maps, $\mu(\A_\eps(q))$
is equal to the measure of the outgoing vectors
based on $\partial B_\eps(q)$ (for the measure
$\cos(\varphi)drd\varphi$), which is equal to
$2\times 2\pi\eps$.
The second assertion follows from~\eqref{eq:const}.
\end{proof}
We first need a result similar to 
Theorem~\ref{thm:rrbilliard}.
\begin{lemma}\label{lem:thm3bis}
Lebesgue almost everywhere we have
$\displaystyle\liminf_{\eps\to0} 
\frac{\log\overline{\mathcal Z}_\eps}{ -\log\eps}\ge 1.$
\end{lemma}
\begin{proof}
We consider again the set $\bar X=\psi(\pi^{-1}\bar M)$ of points in $X$
with previous reflection in $\bar M$.
Let $\alpha>0$ and set
\[
\bar X'_\alpha=\{x=(q,v)\in \bar X \colon d(q,\partial Q)>\alpha\}.
\]
Let $n\ge 1$ be an integer and set $r_n:=\frac{1}{n(\log n)^2}$. 
We define the set $G_n$ of points in $\bar X'_\alpha$ coming
back (modulo $\ZZ^2$) in the $r_n$-neighborhood of the initial position between
 the $n$-th and the $(n+1)$-th reflections by
\[
G_n:=\{x\in\bar X'_\alpha\colon T^{n-1}(\Phi_{\tau(x)}(x))
   \in\bigcup_{\ell\in{\mathbb Z}^2}
   {\mathcal A}_{r_n}(\Pi_Q(x)+\ell)\} .
\]

We take a family of pairwise disjoint open balls $D_i\subset\bar M$ of radius $r_n$ such 
that their union has $\bar\mu$-measure larger than $1-4r_n$. 
As in Section~\ref{sec:proofmain}, we then construct the family $P_{ij}$ following the same procedure. We drop those $P_{ij}$'s not intersecting $\bar X'_\alpha$. For each $i,j$ we fix a point $y_{ij}\in P_{ij}\cap \bar X'_\alpha$.
There exists $L_0>0$ such that for all $x\in X'_\alpha$ we have $\bar\A_{r_n}(\Pi_Q(x))\subset\bar\A_{L_0r_n}(\Pi_Q(y))$ whenever $d(x,y)<r_n$. Thus
\begin{eqnarray*}
{Leb}(G_n)&\le &\sum_{i,j}
{Leb}\left(x\in P_{ij}\colon
   \bar T^{n-1}(\Phi_{\tau(x)}(x))
   \in \bar{\mathcal A}
     _{r_n}(\Pi_Q(x))\right)\\
&\le& \sum_{i,j}r_n\bar\mu\left(D_i\cap\bar T^{-n}
    \left(
  \bar{\mathcal A}_{L_0r_n}(\Pi_Q(y_{ij}))\right)\right).
\end{eqnarray*}
Now, we approximate the indicator
function of $D_i$ by the Lipschitz function
$f_i=\max(1-\frac{d(\cdot,D_i)}{r_n},0)$.
We approximate in the same way the indicator
function of
$\bar{\mathcal A}_{L_0r_n}(\Pi_Q(y_{ij}))$ by a
Lipschitz function $g_{ij}$.
Using the exponential decay of covariance
for Lipschitz functions (Theorem~\ref{thm:young}) we get
\[
\bar\mu(D_i\cap \bar T^{-n}\bar{\mathcal A}_{L_0r_n}(\Pi_Q(y_{ij})))
\le
C\theta^n r_n^{-2} + \int f_i d\bar\mu\int g_{ij} d\bar\mu.
\]
Therefore
\[{Leb}(G_n)
\le C\theta^{n} r_n^{-5}
  +\sum_{i,j}r_n4\bar\mu(D_i)\bar\mu(
     \bar{\mathcal A}_{L_1r_n}(\Pi_Q(y_{ij}))),\]
for some constant $L_1$ (since $\Pi_Q\circ \Phi_s\circ \psi$ is Lipschitz for any $0\le s\le\tau_+$).
According to Lemma~\ref{lem:mesAeps} we get $\sum_{n\ge 1}Leb(G_n)<+\infty$.
Therefore, by the first Borel-Cantelli lemma, for almost every $x\in \bar X'_\alpha$,
there exists $N_x$ such that, for all $n\ge N_x$,
$T^{n-1}(\Phi_{\tau(x)}(x))
   \not\in\bigcup_{\ell\in{\mathbb Z}^2}
   {\mathcal A}_{r_n}(\Pi_Q(x)+\ell)$.
Let 
\[
\varepsilon_0=\min\{d(\Pi_Q(\Phi_s(x)),   \Pi_Q(x)+{\mathbb Z}^2)\colon s\in[\alpha, N_x\tau_+]\}.
\]
We admit temporarily the following result~:
\begin{sub-lemma}\label{lem:sslemma}
The set $\{x\in X\colon \exists s>0,\ \Pi_Q(\Phi_s(x))-\Pi_Q(x)\in{\mathbb Z}^2\}$
has zero Lebesgue measure.
\end{sub-lemma}
Hence $\varepsilon_0$ is almost surely non-null.
Therefore, for almost every point $x$ in $\bar X'_\alpha$,
for all $n\ge N_x$ such that $r_n<\varepsilon_0$,
and all $k=0,...,n$, the point $T^{k-1}(\Phi_{\tau(x)}(x))
\not\in\bigcup_{\ell\in{\mathbb Z}^2}
{\mathcal A}_{r_n}(\Pi_Q(x)+\ell)$ and so
$\overline{\mathcal Z}_{r_n}(x)\ge (n-1)\tau_-$.
Hence 
$\displaystyle\liminf_{n\rightarrow +\infty}
\frac{\log\overline{\mathcal Z}_{r_n}(x)}{-\log r_n}\ge 1$.
Since $\log r_n\sim\log r_{n+1}$, we end up with
$\displaystyle\liminf_{\eps\rightarrow 0}
\frac{\log\overline{\mathcal Z}_{\eps}}
{-\log\eps}\ge 1$ $\mu$-a.e. on $X'_\alpha$. The conclusion follows from $\mu(X'_\alpha)\to1$ as $\alpha\to0$.
\end{proof}
\begin{proof}[Proof of Sub-lemma \ref{lem:sslemma}.]
Let $x$ be a point in $X$ such that, for some $s>0$, 
we have $\Pi_Q(\Phi_s(x))-\Pi_Q(x)\in{\mathbb Z}^2$.
Then either $s<\tau(x)$ which implies that $x$ has a rational direction, or
there exists $n\ge 1$ such that a particle with configuration 
$T^{n-1}(\Phi_{\tau(x)}(x))$ will visit $\Pi_Q(x)+{\mathbb Z}^2$
before the next reflection. We have to prove that the set $C$ of points $x$ satisfying the
second condition has zero Lebesgue measure.
For any $q$ in $Q\setminus\partial Q$, 
we denote by $C_q$ the set of points of $C$ with position $q$.
We have
$$
Leb_X(C\vert \Pi_Q=q)=Leb_Q(C_q)=\int_{\bar T(\overline{\mathcal A}_0(q))\cap{\bar
 T^{-(n-1)}(\overline{\mathcal A}_0(q))}} f_q(r)dr ,
 $$
(for some positive measurable function $f_q$) 
where $\overline{\mathcal A}_0(q)$ is the set of points $m\in\bar M$ that visits
$q+{\mathbb Z}^2$ before the next reflection.
The set $T(\overline{\mathcal A}_0(q))$ is a finite union
of curves $\gamma_{1}$ given by $\varphi=\varphi_1(r)$.
Analogously, the set $\bar
 T^{-(n-1)}(\overline{\mathcal A}_0(q))$ is a finite union of curves
 $\gamma_{-(n-1)}$ given by $\varphi=\varphi_{-(n-1)}(r)$.
 Moreover, each $\gamma_1$ is transversal to each 
$\gamma_{-(n-1)}$
 ($\varphi_1$ is stricly increasing and 
$\varphi_{-(n-1)}$ is strictly decreasing).
Hence the intersection of $\bar T({\mathcal A}_0(q))$ and of $\bar
 T^{-(n-1)}(\overline{\mathcal A}_0(q))$ is finite. 
\end{proof}
Lemma~\ref{lem:thm3bis} enables to prove the following
lemma analogous to Lemma~\ref{lem:density}.
We call 
\[
\bar M_\tau:=\{(m,s)\in \bar M\times \RR\colon 0\le s<\tau(\psi(m))\}.
\] 
\begin{lemma}\label{lem:lem9bis} 
For $\bar\mu$-almost every $(m,s)\in \bar M_\tau$ the following holds:

For any families $(q_\eps)_\eps$ of $Q$, $(D_\varepsilon)_\varepsilon$ of subsets of $\bar M$ such that 

(i) $m\in D_\eps\subset \bar\A_\eps(q_\eps)$ 

(ii) $\Phi_s(\psi(m))\in \bigcup_{\ell\in{\mathbb Z}^2}B(q_\eps+\ell,\eps)\times S^1$

(iii) $D_\varepsilon$ is either a ball or the set ${\bar\A}_\eps(q_\eps)$

we have for all $\alpha>0$
\[
\bar\mu(\bar W_{\bar\A_{\eps}(q_\eps)}\le 
\eps^{-1+\alpha}| D_\eps)\to 0
\quad\text{as $\eps\to0$}.
\]
\end{lemma}
\begin{proof}
We do not detail the proof when 
$D_\eps$ is a ball since it is a direct
adaptation of the proof of Lemma~\ref{lem:density}
with the use of Lemma~\ref{lem:thm3bis} instead
of Theorem~\ref{thm:rrbilliard}.

We suppose that $D_\eps=\bar\A_\eps(q_\eps)$.
The idea is to consider the billiard flow
modulo ${\mathbb Z}^2$ and to adapt the proof of
Lemma~\ref{lem:density} thanks to the Fubini 
theorem.

Let $\alpha>0$ and let $a\in(0,\alpha)$. Let $\eta>0$ and $\eps_0>0$. We set for all $q'$ in $Q$
\[
Bad(q')=\left\{v\in S^1\colon \exists\eps\le\eps_0,\
    \frac{\log\overline{\mathcal Z}_\eps(q',v)}{-\log\eps}
     <1-a\right\}
\]
and
$$
F_\eta(\eps_0)=\{q'\in Q\colon
    {Leb}_{S^1}(Bad(q'))\le\eta\}. 
$$
Let $m$ and $s$ be such that $\Pi_Q(\Phi_s(\psi(m)))$
is a density point 
in $Q$ of the set $F_\eta(\eps_0)$ with respect to the Lebesgue
basis of balls in $Q$. We have
\begin{equation}\label{eq:feta}
\lim_{\eps\to0} \sup_{q_\eps} Leb_Q(F_\eta(\eps_0)^c|B(q_\eps,2\eps))=0
\end{equation}
where the supremum is taken among all the $q_\eps$ satisfying the hypothesis. We observe that $F_\eta(\eps_0)$ is stable by $\ZZ^2$-translations and that,
for all $\eta>0$, $\lim_{\eps_0\rightarrow 0}
   {Leb}_Q((Q\cap [0;1)^2)\setminus F_\eta(\eps_0))=0$.
   Therefore, for a.e. $(m,s)$ and any $\eta>0$ there exists a choice of $\eps_0$ such that $\eqref{eq:feta}$ holds.
Let 
\[
H_\eps:=
(B(q_\eps,2\eps)\times S^1) \cap
 \bigcup_{s\in(6\eps(\tau_+)\eps^{-1+\alpha})}
   \Phi_{-s}\left(\bigcup_{\ell\in{\mathbb Z}^2}
    B(q_\eps+\ell,2\eps)\times S^1\right).
\]
There exists $\eps_1\in(0,\eps_0)$ such that, for all
$\eps\in(0,\eps_1)$, we have
\begin{eqnarray*}
 H_\eps  &\subset&  (B(q_\eps,2\eps)\times S^1)\cap
      \{\overline{\mathcal Z}_{4\eps}\le
       \tau_+\eps^{-1+\alpha}\}\\
   &\subset&
     \{(q',v)\in B(q_\eps,2\eps)\times S^1\colon
        v\in Bad(q')\}.
\end{eqnarray*}
Therefore
\begin{eqnarray*}
Leb_X(H_\eps)&=&Leb_X(\Pi_Q^{-1}(F_\eta(\eps_0))
\cap H_\eps)
+Leb_X(H_\eps\setminus\Pi_Q^{-1}(F_\eta(\eps_0)))\\
&\le&\eta Leb_Q(B(q_\eps,2\eps))+2\pi 
   Leb_Q(B(q_\eps,2\eps)
   \setminus F_\eta(\eps_0)).
\end{eqnarray*}
This together with~\eqref{eq:feta} yields to
\[
\limsup_{\eps\rightarrow 0}
   Leb_X(H_\eps\vert B(q_\eps,2\eps)\times S^1)
   \le\frac{\eta}{2\pi}.
\]
Since $\eta>0$ is arbitrary, for almost every $(m,s)$, we get
\[\lim_{\eps\rightarrow 0}
   Leb_X(H_\eps\vert B(q_\eps,2\eps)\times S^1)
   =0.\]
Hence
\[Leb_X(H_\eps\cap(B(q_\eps,2\eps)\times S^1))
   =o(\eps^2). \]
Moreover, setting $I_s(m)=\mbox{length}\{s\in(0;\tau(m))\colon
 \Phi_s(m)\in B(q_\eps,2\eps)\times S^1\}$
and using the representation of $\Phi_s$
as a special flow over $T$ gives
\begin{eqnarray*}
Leb_X(H_\eps\cap(B(q_\eps,2\eps)\times S^1))
&\ge&
\int_{\bar \A_{2\eps}(q_\eps)
 \cap\{\bar W_{\bar \A_{2\eps}(q_\eps)}\le \eps^{-1+\alpha}\}}
\gluuu I_s(m) \, d\mu(m)\\
&\ge&\int_{\bar \A_{\eps}(q_\eps)
 \cap\{\bar W_{\bar \A_{\eps}(q_\eps)}\le \eps^{-1+\alpha}\}}
\gluuu I_s(m) \, d\mu(m)\\
&\ge&\eps\mu(\bar \A_\eps(q_\eps)\cap\{\bar W_{\bar \A_\eps
   (q_\eps)}\le\eps^{-1+\alpha}\})
\end{eqnarray*}
This finally gives
$$\bar\mu(\bar \A_\eps(q_\eps)\cap\{\bar W_{\bar \A_\eps
   (q_\eps)}\le\eps^{-1+\alpha}\})=o(\eps)
   =o(\bar\mu(\bar \A_\eps(q_\eps))).$$
\end{proof}
We denote by ${\mathcal NS}'$ the set of
couples $(m,s)\in \bar M_\tau$ satisfying the conclusion of Lemma~\ref{lem:lem9bis}.
This is essential for the following lemma analogous to Proposition~\ref{pro:upperbound}
\begin{lemma}\label{lem:prop10bis}
For all $(m,s)\in\mathcal{NS}'$, 
there exists a function $f_{m,s}$
such that $\lim_{\varepsilon\rightarrow 0}
f_{m,s}(\varepsilon)=0$ and such that,
for any families $(q_\eps)_\eps$ of $Q$
and $(D_\eps)_\eps$ of subsets of $M$ such that~:
\begin{itemize}
\item[(i)] $m\in D_\eps\subset {\mathcal A}_\eps(q_\eps) $;
\item[(ii)] $\Phi_s\psi(m)\in B(q_\eps,\eps)\times
S^1$;
\item[(iii)] $D_\eps$ is a ball of radius larger
than $\eps^{1.2}$ or is the set
${\mathcal A}_\eps(q_\eps)$;
\end{itemize}
for all $N\in(e^{\log^2\eps},e^{\frac{1}
{\eps^{2.5}}})$, we have~:
\[
\left\vert \mu(W_{{\mathcal A}_\eps(q_\eps)}
(\cdot)>N|D_\eps) -
\frac{1}{1+\log(N)\bar\mu(\bar{\mathcal A}_\eps
   (q_\eps))\beta}\right\vert\le
 f_{m,s}(\eps)
\]
and
\[
\mu(W_{{\mathcal A}_\eps(q_\eps)}(\cdot)>N
 |{\mathcal A}_\eps(q_\eps)) = 
\frac{1+o_\eps(1)}{1+\log(N)\bar\mu(
\bar{\mathcal A}_\eps(q_\eps))\beta},
\]
where the error term $o_\eps(1)$ is bounded
by $f_{m,s}(\varepsilon)$.
\end{lemma}
\begin{proof}
To simplify the proof, we use the notations
$A={\mathcal A}_\eps(q_\eps)$ and
$\bar A=\bar{\mathcal A}_\eps(q_\eps)$.

\underline{First step}~:
We adapt the proof of Lemma~\ref{lem:upper} 
to prove that
\[
\mu(W_A>N|D)+\beta\log(N)\bar\mu(\bar A)
\mu(W_A>N|A)\le 1+o_\eps(1).
\]
A slight difficulty comes from the fact that
the set $A$ can be divided into several cells.
More precisely, there exist pairwise
disjoint subsets $A_\ell$ of 
$\bar M$ such that (with obvious notations)
\[
A=\bigcup_{\vert\ell\vert\le\tau_+}
  (A_\ell+\ell) \ \mbox{and}
   \ \bar A=\bigcup_{\vert \ell \vert\le \tau_+}
   A_\ell.
\]
Analogously, there exist pairwise
disjoint subsets $ D_\ell$ of 
$\bar M$ such that
\[
D=\bigcup_{\vert\ell\vert\le\tau_+}
  ( D_\ell+\ell).
\]
Hence, we have
\[
\begin{split}
\mu(D) &= \sum_{q=0}^N\mu(D;\ 
   T^{-q}(A;W_A>N-q))\\
   &\ge \mu(D;W_A>N)+ \sum_{q=p_0}^N\mu(D;\ 
   T^{-q}(A;W_A>N))\\
&\ge \mu(D;W_A>N)+ \sum_{q=p_0}^N \sum_{\ell,\ell'}
   \mu(D_{\ell'}+\ell';\ 
   T^{-q}(A_\ell+\ell;W_A>N))\\
&\ge \mu(D;W_A>N)+  \sum_{q=p_0}^N \sum_{\ell,\ell'}
   \mu(D_{\ell'}; S_q\kappa=\ell-\ell';
   \bar T^{-q}(A_\ell;W_{A-\ell}>N)).
\end{split}
\]
This together with~\eqref{eq:const}, 
as in the proof of Lemma~\ref{lem:upper}, give
\[
\mu(D)\ge  \mu(D;W_A>N)+  
\beta\log(N)\frac{\mu(D)}{2\Gamma}\mu( A;W_A>N)+o(\mu(D))
\]
and so
\[
1\ge  \mu(W_A>N\vert D)+  
\beta\log(N){\bar\mu(\bar A)}\mu(W_A>N\vert A)
+o(1).
\]

\underline{Second step}~:
To prove the following lower bound
\[
\mu(W_A>N|D)+\beta\log(N)\bar\mu(\bar A)
\mu(W_A>N|A)\ge 1+o_\eps(1),
\]
we use the notations $m_N$ and $n_N$ of the proof 
of Lemma~\ref{lem:lower} and we write
\[
\begin{split}
\mu(D)&=  \sum_{q=0}^{n_N}\mu(D;\ 
   T^{-q}(A;W_A>n_N-q))\\
&= \mu(D;W_A>N)+\sum_{q=1}^{n_N}
\sum_{\ell,\ell'}
\mu(D_{\ell'};\ S_q\kappa=\ell-\ell';
\bar T^{-q}(A_\ell;W_{A-\ell}>N)).
\end{split}
\]
A first difference with the proof
of lemma~\ref{lem:lower} is that we work
with $D_{\ell'}$ and $A_\ell$ instead of
considering directly $D$ and $A$.
We approximate $D_{\ell'}$ by a set
$D_{\ell'}''$ and $A_\ell$ by a set
$A_\ell''$ as we approximate $D$
by $D''$ in the proof of Lemma~\ref{lem:lower}.

We fix $\alpha\in(0,0.5)$ and
we follow the scheme of the proof of Lemma~\ref{lem:lower} 
for the estimate of $S_0$ and
$S_3$ (using $D_{\ell'}''$ and $A_{\ell}''$). 
We take $M_\eps=\eps^{-1+\alpha}$
instead of $M_\eps=\eps^{2(-1+\alpha)}$.
According to Lemma~\ref{lem:lem9bis},
this choice of $M_\eps$ gives the correct 
estimate of $S_1$. We introduce $M'_\eps=\eps^{-6}$.
We decompose $S_2$ in two blocks~: $S'_2$
is the sum for $q$ in the range $M_\eps+1,...,M'_\eps$
and $S''_2$ in the range $M'_\eps+1,...,m_N$.

To estimate $S'_2$ and $S''_2$, we
approximate
$E_\ell:=A_\ell\cap\{W_{A-\ell}>N\}$ by a set
$E_\ell''$
as we approximate $E$ by $E''$ in the proof of
Lemma~\ref{lem:lower}.

We estimate $S''_2$ as we estimate $S_2$
in the proof of Lemma~\ref{lem:lower}
with $M'_\eps$ instead of $M_\eps$:
$$S''_2\le \log\left(\frac{m_N}{M'_\eps}\right)
  \beta\frac{\mu(D)\mu(E)}{2\Gamma}(1+o(1))+\frac
  {ck\bar\mu(A)^{\frac{1}{p}}}{\sqrt{M'_\eps
  -2k}} +o(\mu(D))$$
and the error term is in $O(\log(\eps)\eps^{1/p}
  \eps^{3})=o(\mu(D))$ provided
$3+1/p>2.4$.

To estimate $S'_2$, we use the symmetry $\pi_0$ on $M$ 
with respect
to the normal $n$ given by~:
$\pi_0(\psi(\ell,i,r,\varphi))=
    \pi_0(\psi(\ell,i,r,-\varphi))$.
Let us notice that $\pi_0$ preserves $\bar\mu$.
Using this symmetry and applying Proposition~\ref{pro:cullt} with $p$ such that 
$1/4>2.4(1-1/p)$, we get
\begin{eqnarray*}
S'_2&\le& 2\Gamma
  \sum_{q=M_\eps}^{M'_\eps}\sum_{\ell,\ell'}
  \bar\mu(D''_{\ell'},S_q\kappa=\ell-\ell';
\bar T^{-q}(A_\ell''))\\
&\le& 2\Gamma
\sum_{q=M_\eps}^{M'_\eps}\sum_{\ell,\ell'}
\bar\mu(\pi_0(A_\ell'');S_q\kappa=\ell'-\ell;
 \bar T^{-q}
(\pi_0(D_{\ell'}'')))\\
&\le&2\Gamma
 \sum_{q=M_\eps}^{M'_\eps}\sum_{\ell,\ell'}
\left[\frac{\beta\bar\mu(A_{\ell}'')\bar\mu(D
     _{\ell'}'')}
   {q-2k}+\frac{ck\bar\mu(D_{\ell'}'')^{1/p}}
    {(q-2k)^{3/2}}\right]\\
&\le&\log\left(\frac{M'_\eps}{M_\eps}\right)
\beta\bar\mu(\bar A)\mu(D)(1+o(1))
   +\frac{c'k\mu(D)^{1/p}(1+o(1))}
    {\sqrt{M_\eps-2k}},
\end{eqnarray*}
the last error term being in 
$O\left(\log(\eps)\bar\mu(D)^{1/p}\eps^{(1-\alpha)/2}
\right)
=o(\bar\mu(D))$ since $(1-\alpha)/2>2.4(1-1/p)$.
Hence, we have proved that, under the assumptions of
Lemma~\ref{lem:prop10bis}, we have
\begin{equation}\label{eq:E1}
\mu(W_A>N\vert D)+\beta\log(N)\bar\mu(\bar A)\mu
(W_A>N\vert A)=1+o(1).
\end{equation}
In the special case $D=A$, we conclude that
\begin{equation}\label{eq:E2}
\mu(W_A>N\vert A)=\frac{1+o(1)}
{1+\beta\log(N)\bar\mu(\bar A)}.
\end{equation}
We turn now to the general case. Applying Equations~\eqref{eq:E1} and~\eqref{eq:E2} we get
$$\mu(W_A>N\vert D)=\frac{1} {1+\beta
     \log(N)\bar\mu(\bar A)} +o(1). $$
\end{proof}

\begin{proof}[Proof of Theorem~\ref{thm:main}-(iii)]

\underline{Upper bound~:}
Let $\bar X_0$ be a set of points of $X$ 
with previous reflection in $\bar M$ and on which
the estimate of Lemma~\ref{lem:prop10bis} is uniform.
Let $\alpha\in(0,1)$ and $\varepsilon_n=\log^{-\alpha}n$. 
Take a cover of $\bar X_0$ by some balls $B(q_n,\frac{\eps_n}2)\times S^1$
for $q_n\in{\mathcal Q}_n\subseteq Q$ such that
$\#{\mathcal Q}_n=O(\eps_n^{-2})$.
We have
\begin{eqnarray*}
Leb(\bar X_0; \Z_{\frac{\eps_n}2}\ge n\tau_+)
&\le&\sum_{q_n}Leb(B(q_n,\frac{\eps_n}2) ; \Z_{\frac{\eps_n}2}\ge n\tau_+)\\
&\le&\sum_{q_n}\varepsilon_n
 \mu(W_{{\mathcal A}_{\frac{\eps_n}2}(q_n)}>n;
    {\mathcal A}_{\frac{\eps_n}2}(q_n))\\
&\le&\varepsilon_n\sum_{q_n}
\mu\left(W_{{\mathcal A}_{\frac{\eps_n}2}(q_n)}>n\Big|
    {\mathcal A}_{\frac{\eps_n}2}(q_n)\right)
    \mu(   {\mathcal A}_{\frac{\eps_n}2}(q_n))\\
&\le&O((1+\beta c\log(n)\varepsilon_n)^{-1})\\
&\le&O((1+\beta c\log^{1-\alpha}(n))^{-1}),
\end{eqnarray*}
with $c=\frac{2\pi}{\Gamma}$ (according to Lemma~\ref{lem:mesAeps}).
Now, by taking $n_k=\exp(k^{2/(1-\alpha)})$ and according
to the Borel-Cantelli lemma,
we get that, for almost all $x$ in $\bar X_0$, there exists $N_x$ such that, for any
$k\ge N_x$, $\Z_{\frac{\eps_{n_k}}2}(x)< {n_k}\tau_+$ and hence
$$\limsup_{k\rightarrow +\infty} \frac{\log\log \Z_{\varepsilon_{n_k}}(x)}{-\log\varepsilon_{n_k}}\le\frac{1}{\alpha}.
$$
Since $\log\varepsilon_{n_k}\sim{\log\varepsilon_{n_{k+1}}}$, we 
conclude that almost everywhere in $\bar X_0$, we have~: 
$$\limsup_{\eps\rightarrow 0} 
\frac{\log\log \Z_{\eps}}{-\log\eps}\le \frac {1}
{\alpha}.$$
Therefore, almost everywhere in $X$, we have
 $$\limsup_{\eps\rightarrow 0} 
   \frac{\log\log \Z_{\eps}}{-\log\eps}\le 1.$$
 
\underline{Lower bound~:}
Let $\bar X$ be the set of points of $X$ with 
previous reflection in $\bar M$. 
Let $\alpha>1$. For all $n\ge 1$, we take
$\eps_n=\log^{-\alpha}n$ and we denote by
$K_n$ the set of points $x\in \bar X$ whose orbit
(by the billiard flow) comes back to the 
$\varepsilon_n$-neighbourhood for the position between
the $n^{th}$ and the $(n+1)^{th}$ reflections~:
$$K_n=\left\{x\in \bar X \colon 
\exists s\in I_n(x) ,\ 
   d(\Pi_Q(x),\Pi_Q(\psi(T^n(\pi\psi^{-1}(x)),s)))<
    \eps_n\right\} ,$$
with $I_n(x):=[0;\tau(T^n(\pi\psi^{-1}(x))))$.
We consider a cover of $\bar X$ by 
sets $C_{\varepsilon_n}(q)=B(q,{\varepsilon_n})
\times S^1$ for $q\in{\mathcal Q}'_n\subseteq Q$ 
such that $\#{\mathcal Q}'_n=O( \eps_n^{-2})$.
Let $n\ge 1$.
For any $q\in{\mathcal Q}_n$, there exist two
families of pairwise
disjoint subsets $ (A_{1,\ell}(q))_\ell $ and 
$(A_{2,\ell}(q))_\ell$
of $\bar M$ such that~:
$$ {\mathcal A}_{\varepsilon_n}(q)
   =\bigcup\left( A_{1,\ell'}(q)+\ell' \right)
\ \ \mbox{and}\ \ \
   {\mathcal A}_{2\varepsilon_n}(q)
   =\bigcup\left( A_{2,\ell}(q)+\ell \right).$$
Let $k$ be such that $\delta^k\approx\eps_n^5$.
Let $A''_{1,\ell'}(q)$ (resp. $A''_{2,\ell}(q)$) be 
the union of all the cylinders $Z\in{\mathcal
  Z}_{-k}^{k}$ intersecting $A_{1,\ell}(q)$ 
(resp. $A_{2,\ell}(q)$). We have~: 
\begin{eqnarray*}
Leb(K_n)&\le&\sum_{q}
Leb(x\in C_{\eps_n}(q)\colon T^n(\pi\psi^{-1}(x))\in 
\A_{2\eps_n}(q))\\
&\le& 2\eps_n \sum_{q} \sum_{\ell,\ell'}
\mu(A_{1,\ell'}(q)+\ell';\ 
T^{-n}(A_{2,\ell}(q)+\ell))\\
&\le& 2\eps_n \sum_{q} \sum_{\ell,\ell'}
\mu(A_{1,\ell'}(q);\ S_n\kappa=\ell-\ell';\ 
\bar T^{-n}(A_{2,\ell}(q)))\\
&\le& 4\eps_n\Gamma
\sum_{q} \sum_{\ell,\ell'}
\bar\mu(A''_{1,\ell'}(q);\ S_n\kappa=\ell-\ell';\ 
\bar T^{-n}(A''_{2,\ell}(q)))\\
&\le& 4\eps_n  \Gamma
\sum_{q} \sum_{\ell,\ell'}
\left[\beta\frac{\bar\mu(A''_{1,\ell'}(q))
\bar\mu(A''_{2,\ell}(q))}{n-2k}+\frac{ck}
    {(n-2k)^{3/2}}\right]\\
&\le& \frac{\eps_n}
\Gamma
\sum_{q}
\left[\beta\frac{\mu(\A_{\eps_n}(q))
\mu(\A_{2\eps_n}(q))}{n-2k}(1+o(1))\right]+
   O(\eps_nn^{-1})\\
&\le& O(\eps_nn^{-1})=O(n^{-1}\log^{-\alpha}n).
\end{eqnarray*}
Hence, according to the first Borel Cantelli lemma,
for almost every $x\in \bar X$,
there exists $N_x$ such that, for all $n\ge N_x$,
for every $s\in I_n(x)$, we have
$$d(\Pi_Q(x),\Pi_Q(\psi(T^n(\pi\psi^{-1}(x)),s)))\ge
    \eps_n.$$
According to Lemma~\ref{lem:sslemma},
\[
u:=\min\left(d(\Pi_Q(x),\Pi_Q(\psi(T^n(\pi\psi^{-1}(x)),s)),\
   n=1,...,N_x,\ s\in I_n(x)\right)
\]
is almost surely non-null.
Therefore, for almost every point $x$ in $\bar X$,
for all $n\ge N_x$ such that $\eps_n<u$,
$\Z_{\eps_n}(x)\ge (n-1)\tau_-$. 
Hence, almost everywhere in $X$, we have
$$\liminf_{n\rightarrow +\infty}\frac{\log\log
    Z_{\eps_n}}{-\log\eps_n}\ge \alpha^{-1}.$$
Since $\log \eps_n\sim\log\eps_{n+1}$, we have
$\liminf_{\eps\rightarrow 0}
\frac{\log\log{\mathcal Z}_{\eps}}
{-\log\eps}\ge \alpha^{-1}$.
Therefore, almost everywhere in $X$, we have
$$\liminf_{\eps\rightarrow 0}\frac{\log\log
    \Z_{\eps}}{-\log\eps}\ge 1.$$
\end{proof}

\begin{proof}[Sketch of proof of Theorem~\ref{thm:main}-(iv)]
This result is obtained by following the same scheme as in the proof of Theorem~\ref{thm:main}-(ii) in Section~\ref{sec:proofmain}. We list the differences:
\begin{itemize}
\item
The set $K\subset \mathcal{NS}\subset \bar M$ is replaced by a set $\mathcal{K}\subset \mathcal{NS}'\subset\bar M_\tau$ such that the convergence in Lemma~\ref{lem:prop10bis} is uniform and such that 
\[
\PP( \psi(\mathcal{K}) )> 1-r.
\]
\item
The family $P_{ij}$:
we first take a family of pairwise disjoint balls $D_i$ of $\bar M$ of radius $\nu_\eps$ such that their union has $\bar\mu$-measure larger than $1-4\nu_\eps$.
We construct the $P_{ij}$'s exactly as in Section~\ref{sec:proofmain}.
Finally we drop the $P_{ij}$'s not intersecting $\mathcal{K}\cap \bar X_r$.
We choose $y_{ij}\in P_{ij}\cap \psi(\mathcal{NS}')$.
\item
The sets $A_{ij}^\pm$ are replaced by 
${\mathcal A}_{ij}^\pm:=\mathcal{A}_{\eps\pm\nu_\eps}(\Pi_Q(y_{ij}))$.
\item
We use the formula for the measure of the ${\mathcal A}_{ij}^\pm$ given by Lemma~\ref{lem:mesAeps}.
\end{itemize}
\end{proof}

%%%%%%%%%%%%%%%%%%%%%%%%%%%%%%%%%%%

\appendix 

\section{Transfer operator and local limit theorem}

\subsection{Hyperbolicity, Young towers and spectral properties of the transfer operator}\label{sec:pf}

We do not repeat the construction of stable and unstable manifolds but only emphasize the hyperbolic estimate that is used throughout the proofs.
Recall that $R_0=\{\varphi=\pm\frac{\pi}{2}\}\subset \bar M$ is the pre-singularity set.
For any $k_1\le k_2$, let $\xi_{k_1}^{k_2}$ be the 
partition of $\bar M\setminus\bigcup_{j=k_1}^{k_2}\bar T^{-j}(R_0)$ into connected components. With a slight abuse of language we will call \emph{cylinders} the elements of $\xi_{k_1}^{k_2}$.

\begin{lemma}\label{lem:bord} 
There exist some constants $c_0$ and $\delta>0$ such that 
for every integer $k$, every set $Z\in \xi_{-k}^k $ has a diameter $\diam Z\le c_0\delta^k$.
\end{lemma}
\begin{proof}
We recall that there exists $C_0>0$ and 
$\Lambda_0>1$ such that, for any
increasing curve contained in a same connected
component of $\xi_{0}^k$, $T^n\gamma$ is 
an increasing curve satisfying
\[
\length(T^n\gamma)\ge C_0\Lambda_0^n{\length(\gamma)}
^2
\]
and such that, for any
decreasing curve contained in a same connected
component of $\xi_{-k}^0$, $T^{-n}\gamma$ is 
a decreasing curve satisfying
\[
\length(T^{-n}\gamma)\ge C_0\Lambda_0^n{\length(\gamma)}
^2.
\]

Let $Z$ be in $\xi_{-k}^k$ and be composed of 
points based on the same obstacle $O_i$. 
The set $Z$ is delimitated
by two increasing curves and two decreasing curves.
Let $m$ and $m'$ be two points in $Z$.
These two points can be joined by a monotonous curve
$\gamma$ in $Z$.

If the curve $\gamma$ is increasing,
then we have
\[
\length(\gamma)\le \sqrt{\frac {\length(T^n\gamma)}
   {C_0\Lambda_0^n}}
\le 
\sqrt{\frac {\pi+\vert\partial
   O_i\vert}{C_0\Lambda_0^n}}.
\]

If the curve $\gamma$ is decreasing,
then, considering $\bar T^{-n}\gamma$, we get
$\length(\gamma)\le\sqrt{\frac {\pi+\vert\partial
   O_i\vert}{C_0\Lambda_0^n}}$.
\end{proof}

We do not repeat the construction of the tower but only briefly recall its property and then introduce the Banach space suitable for the study of the transfer operator.
Young constructed in~\cite{young} two dynamical systems $(\tilde M,\tilde T, \tilde\mu)$ and $(\hat M, \hat T, \hat \mu)$ such that there exist two measurable functions $\tilde\pi\colon \tilde M\to \bar M$ and $\hat \pi\colon \tilde M\to \hat M$ such that $\tilde\pi\circ \tilde T=\bar T\circ \tilde\pi$, $\tilde\pi_*\tilde\mu=\bar\mu$, $\hat\pi\circ \tilde T=\hat T\circ \hat\pi$, $\hat\pi_*\tilde\mu=\hat\mu$.

These dynamical systems are towers and are such for any measurable $f\colon \bar M\to\CC$ constant on each stable manifold there exists $\hat f\colon \hat M\to\CC$ such that $\hat f\circ \hat\pi=f\circ \tilde\pi$.
For each $\ell\ge0$, we denote by $\hat\Delta_\ell$ the $\ell$th floor of the tower $\hat M$. This $\ell$-floor is partitioned in $\{\hat\Delta_{\ell,j}\colon j=1,\ldots,j_\ell\}$. The partition $\D=\{\hat\Delta_{\ell,j}\colon \ell\ge0,\ j=1,\ldots,j_\ell\}$ is Markov.
For any $x,y$ belonging to the same atom of $\D$, we define
\[
s(x,y):=\max\{n\ge 0\colon \forall i\le n,\ \D(\hat T^ix)=\D(\hat T^iy)\}.
\]
For any such $x,y$, the sets $\tilde\pi\hat\pi^{-1}\{x\}$ and $\tilde\pi\hat\pi^{-1}\{y\}$
are contained in the same connected component of $\bar M\setminus\bigcup_{k=0}^{s(x,y)}\bar T^{-k}R_0$.

Let $p>1$ and set $q$ such that $\frac1p+\frac1q=1$. Let $\eps>0$ and $\beta\in(0,1)$ well chosen. Young defines for $\hat f\in L^q_\CC(\hat M,\hat\mu)$
\[
\|\hat f\| = \sup_{\ell} \|\hat f_{|\hat \Delta_\ell}\|_\infty e^{-\ell\eps}
+
\sup_{\ell,j} \esssup_{x,y\in \hat\Delta_{\ell,j}} \frac{|\hat f(x)-\hat f(y)|}{\beta^{s(x,y)}}e^{-\ell \eps}.
\]
Let $\V=\{\hat f\in L^q_\CC(\hat M,\hat\mu)\colon \|\hat f\|<\infty\}$.
This defines a Banach space $(\V,\|\cdot\|)$, such that $\|\cdot\|_{q}\le \|\cdot\|$.
Let $P$ be the Perron-Frobenius operator on 
$L^q$ defined as the adjoint of the composition by $\hat T$ on $L^p$.
This operator $P$ is quasicompact on $\V$. The construction of the tower can be adapted in such a way that its dominating eigenvalue on $\V$ is $1$ and is simple. This choice will be convenient for our proof and we will adopt it, although it is not essential.

The cell shift function $\kappa$ is centered in the sense that
\[
\int \kappa d\bar\mu = 0
\]
and its asymptotic covariance matrix
\begin{equation}\label{eq:cov}
\Sigma^2 := \lim_{n\to\infty} \frac{1}{n} \cov_{\bar\mu}(S_n\kappa)
\end{equation}
is well defined and non-degenerated.
Since $\kappa\colon \bar M\to\ZZ^2$ is constant on the local stable manifolds, there exists $\hat \kappa\colon\hat M\to\ZZ^2$ such that $\hat \kappa\circ \hat \pi=\kappa\circ\tilde\pi$. For any $u\in\RR^2$, we define $P_u(\hat f)=P(e^{i u\cdot\hat\kappa}\hat f)$.
The method introduced by Nagaev~\cite{nag1,nag2} and developed by Guivarc'h and Hardy~\cite{gui} and many other authors has been applied in this context by Sz\'asz and Varj\'u~\cite{SV} (see also~\cite{pene1}). They have established the following key result:
\begin{proposition}\label{pro:pertu}
There exist a real $a\in(0,\pi)$, a $C^3$ family of complex numbers $(\lambda_u)_{u\in[-a,a]^2}$, 
two $C^3$ families of linear operators on $\V$: $(\Pi_u)_{u\in[-a,a]^2}$ and $(N_u)_{u\in[-a,a]^2}$ such that
\begin{itemize}
\item[(i)]
for all $u\in[-a,a]^2$ we have 
$P_u^n=\lambda_u^n \Pi_u + N_u^n$;
Moreover $\Pi_0\hat f=\int_{\hat M}\hat f\, d\hat\mu$ for any $\hat f\in L^q$;
\item[(ii)]
there exists $\nu\in(0,1)$ such that 
\[
\sup_{u\in[-a,a]^2} \|| N_u^n\|| = O(\nu^n)
\quad\text{and}\quad
\sup_{u\in[-\pi,\pi]^2\setminus[-a,a]^2} \|| P_u^n\|| = O(\nu^n);
\]
\item[(iii)]
we have $\lambda_u = 1-\frac12\Sigma^2u\cdot u = O(|u|^3)$;
\item[(iv)] there exists $\sigma>0$ such that,
for any $u\in[-a,a]^2$, 
$|\lambda_u|\le e^{-\sigma|u|^2} $ and
$e^{-\frac 12\Sigma^2u\cdot u}\le e^{-\sigma|u|^2} $.
\end{itemize}
\end{proposition}

Note that by taking $u=0$ in the proposition we recover the estimate on the rate of decay of correlations below. We state it here in a form suitable for our purpose, in particular to prove the results of Section~\ref{sec:recbilmap}.
\begin{theorem}[\cite{young}]\label{thm:young} There exist some constants $C>0$ and $\theta\in(0,1)$ such that for all Lipschitz functions $f$ and $g$ from $\bar M$ to $\RR$, 
\begin{equation}\label{eq:covariance}
\int f\circ T^n g d\bar\mu-\int f d\bar\mu\int 
g d\bar\mu\le C\theta^n \|f\|_{Lip}\|g\|_{Lip}.
\end{equation}
Moreover, if $f$ is the indicator function of a union of components of $\xi_{-k}^k$ and $g$ is the indicator function of a union of components of $\xi_{-k}^{+\infty}$ then the covariance in~\eqref{eq:covariance} is simply bounded by $C\theta^{n-2k}$.
\end{theorem}

However this information is not sufficient to control the recurrence for the extended billiard map $T$, therefore we need a finer version.

\subsection{Conditional uniform local limit theorem}~\label{sec:cullt}

Here we prove the local limit theorem, 
Proposition~\ref{pro:cullt}, concerning the billiard 
map $\bar T$ and its $\ZZ^2$-cocycle $S_n\kappa$.
\medskip

\begin{proposition41} %Proposition~\ref{pro:cullt}}
Let $p>1$. There exists $c>0$ such that, 
for any $k\ge 1$, if $A\subset\bar M$ is a union of components of $\xi_{-k}^{k}$ and $B\subset \bar M$ is a union of $\xi_{-k}^\infty$ then for any $n > 2k$ and $\ell\in\ZZ^2$
\[
\left|
\bar\mu(A\cap \{S_n\kappa=\ell\}\cap \bar T^{-n}(B))
-\frac{\beta e^{-\frac1{2(n-2k)} (\Sigma^2)^{-1}\ell\cdot\ell}}{(n-2k)}\bar\mu(A)\bar\mu(B)\right|
\le \frac{ck\bar\mu(B)^{\frac1p}}{(n-2k)^\frac32}
\]
where $\beta=\frac{1}{2\pi\sqrt{\det \Sigma^2}}$.
\end{proposition41}

\begin{proof}
The set $\bar T^{-k}A$ is a union of components of $\xi_{0}^{2k}$ and $\bar T^{-k}B$ is a union of components of $\xi_0^\infty$.
Let $\hat A=\hat\pi(\tilde\pi^{-1}\bar T^{-k}A)$ and $\hat B=\hat\pi(\tilde\pi^{-1}\bar T^{-k}B)$.
Note that $\tilde\pi^{-1}\bar T^{-k}A=\hat\pi^{-1}\hat A$ and $\tilde\pi^{-1}\bar T^{-k}B=\hat\pi^{-1}\hat B$.
Setting 
\[
C_n(A,B,\ell):=\bar\mu(A ;S_n\kappa=\ell ; 
\bar T^{-n}B),
\]
we have
\[
\begin{split}
C_n(A,B,\ell)
&=
\int_{\hat M} 1_{\hat A} 1_{\{S_n\hat\kappa=\ell\}}\circ \hat T^k 1_{\hat B}\circ \hat T^n\, d\hat\mu\\
&=
\int_{\hat M} P^k(1_{\hat A}) 1_{\{S_n\hat\kappa=\ell\}} 1_{\hat B}\circ \hat T^{n-k}\, d\hat\mu\\
&=
\frac{1}{(2\pi)^2}\int_{[-\pi,\pi]^2}\gluu e^{-iu\cdot\ell}
\underbrace{ 
\int_{\hat M} P^k(1_{\hat A}) e^{iu\cdot S_n\kappa}1_{\hat B}\circ \hat T^{n-k}\, d\hat\mu}_{a(u)} du.
\end{split}
\]
We have 
\[
\begin{split}
a(u)&=\int_{\hat M} P_u^n(P^k(1_{\hat A})1_{\hat B}\circ \hat T^{n-k})\, d\hat\mu\\
&=\int_{\hat M} P_u^k(1_{\hat B} P_u^{n-k}P^k(1_{\hat A}))\, d\hat\mu\\
&=\int_{\hat M} P_u^k(1_{\hat B} P_u^{n-2k}(b_u^k))\, d\hat\mu,
\end{split}
\]
with $b_u^{k}:=P_u^{k}P^k(1_{\hat A})$. Set 
\[
a_1(u):=\int_{\hat M} P^k(1_{\hat B} P_u^{n-2k}(b_u^{k}))\, d\hat\mu.
\]
We have, since $\||P_u^k-P^k\||_{L^1\to L^1}\le 
 |u| k \|\kappa\|_\infty$,
\[
\begin{split}
\left|a(u)-a_1(u)\right|
&\le \||P_u^k-P^k\||_{L^1\to L^1}
\int_{\hat M} 1_{\hat B} |P_u^{n-2k}(b_u^{k})|\, d\hat\mu\\
&\le\|\kappa\|_\infty k |u|\|P_u^{n-2k}b_u^{k}\| \hat\nu(\hat B)^{1/p}
\end{split}
\]
by the H\"older inequality and since the norm 
$\|\cdot\|$ dominates the $L^q$ norm.
Let us notice that by the Markov property $\sup_{u\in[-\pi,\pi]^2}\|b_u^{k}\|=O(1)$, uniformly in $A$ and $k$. We have by Proposition~\ref{pro:pertu} (i) and (ii)
\[
\frac{1}{(2\pi)^2}\int_{[-\pi,\pi]^2}|u|\||P_u^{n-2k}\|| du =
\int_{[-a,a]^2}|u||\lambda_u|^{n-2k}du+O(\nu^{n-2k}).
\]
In addition, by Proposition~\ref{pro:pertu} (iv) we have
\begin{equation}\label{eq:bouli}
\int_{[-a,a]^2}\gluu |u||\lambda_u|^{n-2k}du
\le
\frac{1}{(n-2k)^{\frac32}}\int_{\RR^2}|v|e^{-\sigma|v|^2}dv
=
O\left(\frac{1}{(n-2k)^{\frac32}}\right),
\end{equation}
with the change of variable $v=\sqrt{n-2k}u$. 
Therefore 
\[
C_n(A,B,\ell) = \frac{1}{(2\pi)^2}
 \int_{[-\pi,\pi]^2} e^{-iu\cdot \ell}
\int_{\hat M} 1_{\hat B} P_u^{n-2k}(b_u^k) \, 
d\hat\mu\, du + O\left(\frac{k\bar\mu(B)^{\frac1p}}{(n-2k)^\frac32}\right)
\]
By the H\"older inequality and since the norm 
$\|\cdot\|$ dominates the $L^q$ norm and according
to points (i) and (ii) of proposition~\ref{pro:pertu},
we have~:
\[
C_n(A,B,\ell) = \frac{1}{(2\pi)^2}
 \int_{[-a,a]^2} e^{-iu\cdot \ell}
\int_{\hat M} 1_{\hat B} \lambda_u^{n-2k}\Pi_u(b_u^k) 
\, 
d\hat\mu\, du + O\left(\frac{k\bar\mu(B)^{\frac1p}}
{(n-2k)^\frac32}\right).
\]
We will use here and thereafter the notation $f_u=O(g_u)$ to mean that there exists some constant $c_*$ such that for all $u\in[-a,a]^2$, we have $|f_u|\le c_* |g_u|$.
The differentiability of $u\mapsto \Pi_u$ gives $|\Vert \Pi_u-\Pi_0|\Vert=O(\vert u\vert)$.
Hence using formula~\eqref{eq:bouli}, we get
\[
C_n(A,B,\ell) = \frac{1}{(2\pi)^2}
 \int_{[-a,a]^2} e^{-iu\cdot \ell}\lambda_u^{n-2k} \, du
\hat\mu(\hat B) \int_{\hat M} b_u^k
\, d\hat\mu + O\left(\frac{k\bar\mu(B)^{\frac1p}}
{(n-2k)^\frac32}\right).
\]
For any $u\in[-a,a]^2$, we have
\[
\int_{\hat M} b_u^k \, d\hat\mu
= \int_{\hat M}e^{iu\cdot S_k\hat\kappa}
   P^k({\bf 1}_{\hat A})\, d\hat\mu= \int_{\hat M}e^{iu\cdot S_k\hat\kappa}\circ\hat T^k
     {\bf 1}_{\hat A}\, d\hat\mu= \hat\mu(\hat A)+O(\vert u\vert).
\]

Again, using formula~\eqref{eq:bouli} we have
\[
C_n(A,B,\ell) = \frac{1}{(2\pi)^2}
 \int_{[-a,a]^2} e^{-iu\cdot \ell}\lambda_u^{n-2k} \, du
\bar\mu(B) \bar\mu( A)
+ O\left(\frac{k\bar\mu(B)^{\frac1p}}
{(n-2k)^\frac32}\right).
\]
According to the point (iii) of 
proposition~\ref{pro:pertu}, we have
\[
\left\vert \lambda_u^{n-2k}
   - e^{-\frac{n-2k}2\Sigma^2u\cdot u} \right\vert
   \le c_*(n-2k) e^{-\sigma\vert u\vert^2(n-2k-1)}O(\vert u\vert^3).
\]
Hence, proceeding similarly as in formula~\eqref{eq:bouli}, we get
\[
\begin{split}
C_n(A,B,\ell) &= \frac{\bar\mu(B) \bar\mu( A)}{(2\pi)^2}
 \int_{[-a,a]^2} e^{-iu\cdot \ell}
e^{-\frac{n-2k}2\Sigma^2u\cdot u} 
 \, du
+ O\left(\frac{k\bar\mu(B)^{\frac1p}}
{(n-2k)^\frac32}\right)\\
&= \frac{\bar\mu(B) \bar\mu( A)}{(2\pi)^2(n-2k)}
 \int_{{\mathbb R}^2} e^{-i\frac{v\cdot \ell}
{\sqrt{n-2k}}}
e^{-\frac{1}2\Sigma^2v\cdot v} 
 \, dv
+ O\left(\frac{k\bar\mu(B)^{\frac1p}}
{(n-2k)^\frac32}\right)
\end{split}
\]
with the change of variable $v=u\sqrt{n-2k}$. Finally, using the formula of the characteristic function of a gaussian, we get
\[
C_n(A,B,\ell)=
\frac{\bar\mu(A)\bar\mu(B)}{(2\pi)^2(n-2k)} 2\pi\sqrt{\det(\Sigma^2)^{-1}}e^{-\frac{(\Sigma^2)^{-1}
  \ell\cdot\ell}{2(n-2k)}}
+ O\left(\frac{k\bar\mu(B)^{\frac1p}}
{(n-2k)^\frac32}\right),
\]
which proves the result after obvious simplifications.
\end{proof}

%%%%%%%%%%%%%%%

\end{document}